\documentclass[reqno,11pt]{amsart}

\usepackage{amsmath}
\usepackage{esint}
\usepackage{amsthm}
\usepackage{calligra}
\usepackage{epsfig}
\usepackage{psfrag}
\usepackage{graphicx}
\usepackage{graphpap,latexsym,epsf}
\usepackage{color}
\usepackage{amssymb,mathrsfs,enumerate}
\usepackage{endnotes}
\usepackage{calligra}
\usepackage{a4wide}
\usepackage[colorlinks, citecolor=citegreen, linkcolor=refred]{hyperref}
\usepackage{enumitem}
\usepackage{accents}
\usepackage[colorinlistoftodos]{todonotes}

\usepackage{mathtools}
\mathtoolsset{showonlyrefs}

\newcommand{\R}{\mathbb{R}}

\newcommand{\Z}{\mathbb{Z}}
\newcommand{\HH}{\mathbb{H}}

\newcommand{\SSS}{\mathbb{S}}

\def\HHH{{\rm H}}

\newcommand{\pa}{\partial}

\newcommand{\Ric}{{\rm Ric}}

\newcommand{\Ro}{{\rm R}}

\newcommand{\na}{\nabla}

\newcommand{\vertl}{\vert\hspace{.07em}}
\newcommand{\vertr}{\hspace{.07em}\vert}

\def\ringg#1{\accentset{\circ}{#1}}

\newcommand{\wh}{\widehat}
\newcommand{\ol}{\overline}

\newcommand{\Hyp}{\mathbb{H}}

\mathchardef\emptyset="001F

\definecolor{citegreen}{rgb}{0,0.6,0}
\definecolor{refred}{rgb}{0.8,0,0}

\theoremstyle{plain}

\newtheorem{theorem}{Theorem}[section]
\newtheorem{corollary}[theorem]{Corollary}
\newtheorem{proposition}[theorem]{Proposition}
\newtheorem{lemma}[theorem]{Lemma}

\theoremstyle{definition}

\theoremstyle{remark}
\newtheorem{remark}[theorem]{Remark}

\newtheorem{example}[theorem]{Example}

\numberwithin{equation}{section}

\makeatletter
\@namedef{subjclassname@2020}{\textup{2020} Mathematics Subject Classification}
\makeatother

\title[Green functions and a PMT for asymptotically
hyperbolic 3-manifolds]{Green functions and a positive mass theorem for asymptotically
hyperbolic 3-manifolds}

\author[K.~Kr\"{o}ncke]{Klaus Kr\"{o}ncke}
\address{Institutionen f\"{o}r Matematik, KTH Stockholm, Lindstedtsv\"{a}gen 25, 10044 Stockholm, Sweden}
\email{kroncke@kth.se}

\author[F.~Oronzio]{Francesca Oronzio}
\address{Scuola Superiore Meridionale, Largo S. Marcellino, 10, 80138, Napoli, Italy}
\email{f.oronzio@ssmeridionale.it}

\author[A.~Pinoy]{Alan Pinoy}
\address{Département de Mathématique, Université Libre de Bruxelles, Bruxelles, Belgique}
\email{alan.pinoy@ulb.be}

\begin{document}

\subjclass[2020]{53C21, 31C12, 53C24, 53Z05.}

\keywords{monotonicity formulas; Green function; Riemannian $3$--manifold with scalar curvature lower bound; volume-renormalized mass.}

\begin{abstract}
We prove a new positive mass theorem for three-dimensional manifolds which are asymptotically hyperboloidal of order greater than $1$.
The mass quantity under consideration is the volume-renormalized mass recently introduced in a paper by Dahl, McCormick and the first author. The proof is based on a monotonicity formula holding along the level sets of the Green function for the Laplace operator centered at an arbitrary point. In order for this argument to work out, we require that the second homology of the manifold does not contain any spherical classes.
\end{abstract}

\maketitle

\section{Introduction}

One of the most fundamental statements in mathematical general relativity is the positive mass theorem. Its standard form states that the ADM-mass
\begin{align*}
m_{ADM}(g)\,=\,\lim_{r\to\infty} \int\limits_{\partial B_r}\!\langle\mathrm{div}_{\wh{g}}g-d\mathrm{tr}_{\wh{g}}g, \nu_{\,\wh{g}}\rangle_{\wh{g}}\,\, d\mathcal{H}^{n}_{\wh{g}}
\end{align*}
 of an asymptotically Euclidean manifold $(M^{n+1},g)$ (with respect to the Euclidean metric $\wh{g}=g_{eucl}$ for a given asymptotically Euclidean chart of order greater than $(n-1)/2$) of nonnegative scalar curvature is nonnegative and zero if and only if the manifold is isometric to Euclidean space $(\R^{n+1},g_{eucl})$.

The first proofs of this result were given by Schoen and Yau in dimensions
$n+1\leq 7$, using minimal hypersurfaces \cite{SY79} and by Witten on spin manifolds using harmonic spinors \cite{Wit81}. The general case remained open for
a couple of decades until Schoen-Yau \cite{SY22} and Lohkamp \cite{Loh16} independently
announced proofs of the positive mass theorem without any additional conditions.

In recent years, further approaches have been used to give different proofs of the positive mass theorem in dimension three, for example by using Ricci flow  \cite{Li18} or  linearly growing harmonic functions \cite{BKKS22}. While the first of these papers uses the well-developed theory of three-dimensional Ricci flows with surgery, the second paper uses the Gauss-Bonnet formula for three-dimensional manifolds in a crucial way.

Agostiniani, Mazzieri and the second author established yet another method to prove the positive mass theorem for three-dimensional manifolds. In \cite{AMO24}, they established a monotonicity formula along the level sets of an appropriate harmonic function related to the fundamental solution of the Laplace operator centered at an arbitrary point $o$. 
The quantity is increasing from $0$ (for level sets near $o$) to a limit at infinity which turned out to be bounded above by a positive multiple of $m_{ADM}$ and the positive mass theorem follows.

In this paper, we use a similiar approach to prove a new positive mass theorem for the volume-renormalized mass $m_{VR}$ on three-dimensional asymptotically hyperboloidal  manifolds. The quantity $m_{VR}$ was recently introduced by Dahl, McCormick and the first author in \cite{DaKrMc}.
Let $M^{n+1}$ be the interior of a manifold $\overline{M}$ with boundary $\partial M$ and $g$ a complete Riemannian metric on $M$.
We call $(M,g)$ asymptotically hyperboloidal of order $\delta>0$ if there is a compact subset $K\subset M$, a large radius $R$ and a diffeomorphism $\varphi:\overline{M}\setminus K\to \overline{\Hyp}^{n+1}\setminus \overline{B}_R$ such that
\begin{equation}\label{decaymetricatinfinity}
\big\|e^{\delta r}(\varphi_{*}g-g_{hyp})\big\|_{C^{2,\alpha}(\Hyp^{n+1}\setminus \overline{B}_R,\,g_{hyp})}<+\infty.
\end{equation}
Here,
 $r$ is the distance to the origin measured with respect to $g_{hyp}$. Moreover, the chart $(M\setminus K, \varphi)$ is called asymptotically hyperboloidal of order $\delta>0$, and the map $\varphi$ is said to be asymptotically hyperboloidal of order $\delta>0$.
\begin{remark}
In view of the Poincar\'{e} ball model of hyperbolic space, we think of $\overline{\Hyp}^{n+1}$ as the closure of the unit ball on $\R^{n+1}$.
\end{remark}
With $M_r=M\setminus \varphi^{-1}\big( \Hyp^{n+1}\setminus B_r\big)$ 
and $\wh{g}=g_{hyp}$, we now define the volume-renormalized mass as
\begin{align*}
m_{VR}(g)\,=\,\lim_{r\to+\infty}\Bigg[\, \int\limits_{\partial B_r}\!\big \langle\mathrm{div}_{\wh{g}} (\varphi_*g)-d\mathrm{tr}_{\wh{g}}(\varphi_*g), \nu_{\,\wh{g}}\big\rangle_{\wh{g}}\,\,d\mathcal{H}^{n}_{\wh{g}}\,+\,2n\Bigg(\,\int\limits_{M_r}\,d\mu_{g}-\int\limits_{B_r}\,d\mu_{\wh{g}}\,\Bigg)\Bigg]\,.
\end{align*}
Fundamental properties of this quantity have been investigated in \cite{DaKrMc}. It is finite and independent of $\varphi$ if $(M,g)$ is asymptotically hyperboloidal of order $\delta>\frac{n}{2}$ and if $\Ro+n(n+1)$ is integrable over $(M,g)$. On the other hand, it is always equal to $+\infty$ if we replace the assumption of integrability of $\Ro+n(n+1)$ with the condition that it is nonnegative and not integrable. Here, $\Ro$ denotes the scalar curvature of $(M,g)$.
Importantly, the quantity $m_{VR}$ can be regarded as a linear combination of the ADM boundery integral and the renormalized volume. Note that for $\frac{n}{2}<\delta<n$, the separate limits of these two terms do in general diverge but the linear combination given here converges. For $\delta>n$, the ADM boundary integral vanishes in the limit, and $m_{VR}$ reduces to a positive multiple of the renormalized volume. 

It is worth pointing out that in \cite{DaKrMc}, $m_{VR}$ was introduced for a much more general class of asymptotically hyperbolic manifolds, allowing arbitrary conformal boundaries and replacing $g_{hyp}$ with a large class of possible reference metrics $\wh{g}$. In order to keep the presentation in the introduction simple, we restrict to the case where the conformal boundary is the round sphere.

\begin{remark}\label{importantpropertyrenvol}
The definition in \cite{DaKrMc} assumes that $(M, g)$ is a conformally compact of at least $C^{2,\alpha}$-regularity and considers an exhaustion of $M$ given by a family of precompact open sets $M_r$ determined by a
boundary defining function. If the order satisfies the inequality $\delta\leq 2$, the conformal compactification will in general have lower regularity. However, going through the proof of \cite[Theorem 3.1]{DaKrMc}, one sees that the above setting and conditions  are good enough to make sure that $m_{VR}(g)$ is well-defined. Moreover, it is not hard to see that the limit does not change if one replaces $M_r$ by an arbitrary exhaustion of $M$ of precompact open subsets with smooth boundary.
\end{remark}

There have been other mass invariants previously defined for asymptotically hyperboloidal manifolds \cite{CN01,Wang01}. The
quantity $m_{VR}$ appears to be significant for the following two reasons. Using it as a normalization, we could in \cite{DaKrMc} for the first time establish a Einstein-Hilbert action for asymptotically hyperbolic manifolds in
a mathematically clean way. Moreover, it seems that  $m_{VR}$ can also be motivated from the Hamiltonian perspective, similarly to the ADM-mass, by considering expanding spacetimes asymptotic to the Milne model  $(\R_+ \times \R^n , -dt^2 + t^2 g_{hyp} )$, instead of asymptotically Minkowskian spacetimes. This is subject to further research.

The main result of this paper is the following positive mass theorem:
\begin{theorem}\label{thm:PMT}
Let $(M,g)$ be an orientable three-dimensional manifold which is asymptotically hyperboloidal of order $\delta>1$. Assume that its scalar curvature satisfies $\Ro\geq -6$ and that its second integral homology $H_2(M;\Z)$ does not contain any spherical classes. Then,
\begin{align*}
m_{VR}(g)\,\geq\, 0
\end{align*}
and equality holds if and only if $(M,g)$ is isometric to $(\Hyp^3,g_{hyp})$.
\end{theorem}
This theorem is a significant generalization of \cite[Theorem D]{DaKrMc}, where the assertion has only been proven under the assumption that $M$ is diffeomorphic to $\Hyp^3$.

\begin{remark}
We observe that there are not non-separating spheres in the manifold $M$ if its second integral homology $H_2(M;\Z)$ does not contain any spherical classes. These topological conditions--concerning the second integral homology group or the absence of non-separating spheres--are not new: the first appeared in \cite{MW24}, while the second in a previous work \cite{Stern}. They are equivalent if the manifold is one-ended (i.e.,  if there is only one unbounded connected component in the complement of any compact subset).

It is a standard fact of three-dimensional differential topology that an orientable $3$-manifold $M$ contains a non-separating sphere if and only if it is diffeomorphic to a connected sum $N\#(\SSS^1\times \SSS^2) $ (see, for instance, \cite[Lemma 3.8 and Lemma 3.16]{Hembook}). Therefore, the decomposition of $M$ into prime manifolds does not contain an $\SSS^1\times \SSS^2$. On the other hand, it is easy to see that for any closed orientable three manifold $N$ with no $\SSS^1\times \SSS^2$, there is a metric $g$ on $M=N\setminus\left\{p\right\}$ such that $(M,g)$ satisfies the assumption of the theorem.
\end{remark}

A crucial ingredient of Theorem \ref{thm:PMT} is the following monotonicity formula along the level sets of a fundamental solution.

\begin{theorem}\label{thm_monotonicity}
Let $(M,g)$ be a complete, noncompact, orientable, three-dimensional Riemannian manifold with scalar curvature greater than or equal to $-6$ and such that its second integral homology $H_2(M;\Z)$ does not contain any spherical classes. We assume that there exists the minimal positive Green function $\mathcal{G}_o$ for $\Delta_g$ with a pole at some point $o \in M$, and that $\mathcal{G}_o$ vanishes at infinity. We consider the function 
\begin{equation}\label{definition_u}
u\,=\,1-4\pi \mathcal{G}_o\,, 
\end{equation}
and let $F:(0, + \infty)\to \R$ be the function defined as 
\begin{align}
F(t)\,\, &= \,\, 4\pi t \,\,+\,\,\, \sinh^3t \cosh t\!\!\!\int\limits_{\{u=2-\coth t\}}\!\!\!\!\! \vert \nabla u \vert^{2} \,\, d\mathcal{H}^2 \, \, - \,\,\sinh^2t \!\!\! \int\limits_{\{u=2-\coth t\}}\!\!\!\!\!\vert \nabla u \vert\, \mathrm{H}   \,\, d\mathcal{H}^2 \\
&\quad \, \,\,+\,\,3\!\!\!\int\limits_{\{u<2-\coth t\}}\!\!\!\frac{\vert \nabla u \vert}{(2-u)^2-1} \,\, d\mu\,\,-\!\!\int\limits_{\{u<2-\coth t\}}\!\!\!\frac{\vert \nabla u \vert^3}{[(2-u)^2-1]^3} \,\, d\mu\,,\label{eq0}
\end{align}
where 
$\mathrm{H}$ is the mean curvature of the level set $\{u=2-\coth t \}\setminus \{\vert \nabla u \vert=0\}$ computed with respect to the $\infty$--pointing unit normal vector field $\nu={\nabla u}/{\vert \nabla u \vert}$.
Then, we have that 
\begin{equation*}
0 < s \leq t < +\infty \quad \Rightarrow \quad  F(s) \, \leq  \, F(t) \, ,
\end{equation*}
provided $2- \coth s$ and $2- \coth t$ are regular values of $u$. 
\end{theorem}

\begin{remark}
The function $F(t)$ defined in \eqref{eq0} shares similarities with the monotone quantity in \cite{AMO24} for the asymptotically Euclidean setting. The first term is identical, the other two terms in the first line appear with different functions in front of the integrals in \cite{AMO24}, due to the fact that the level sets are defined by the analogous function on hyperbolic space. The terms in the second line are new and compensate the renormalized volume part in $m_{VR}$. 
\end{remark}

\begin{example} 
The class of Riemannian manifolds that satisfy the assumptions of Theorem \ref{thm_monotonicity} is clearly larger than that of Theorem \ref{thm:PMT}. They could for example be asymptotically hyperbolic with a conformal boundary different than the sphere. In this case, we do not know what the limit $\lim_{t\to\infty}F(t)$ will be.
\end{example}

In order to compute the limit of the monotone function, we need to exploit the asymptotic expansion of the function $u$, which is directly determined by that of Green function $\mathcal{G}_o$ (as $u=1-4\pi \mathcal{G}_o$). 
Therefore, under the additional assumption that there exists a distinguished asymptotically hyperboloidal map of order $\delta>1$, in which the function $\mathcal{G}_o$ admits expansion
\begin{equation}\label{eq:polyhom_expansion_prelllbis}
\varphi_*\mathcal{G}_o(r,\xi)\,=\,\phi(\xi) \,e^{-2r}+O_2(e^{-3r})\,,
\end{equation}
where $\phi$ is a smooth and positive function on $\SSS^2$, we are able to show that
\begin{equation}\label{thm_limit}
\lim_{t\to +\infty} F(t)\,\leq\, \frac{1}{2}\, m_{VR}(g)\,.
\end{equation}
Thus, by combining Theorem \ref{thm_monotonicity} with the limits of the function $F(t)$ as $t\to 0^+$ and $t\to +\infty$, we establish the non-negativity of $m_{VR}(g)$. 
A density argument will then be used to get rid of this additional assumption and prove the theorem as stated.
\medskip

The paper is organized as follows: In Section \ref{SectPreliminaries}, we recall and discuss some preliminary material.
In Section \ref{SectEst}, we present the asymptotic behavior at infinity of the minimal positive Green function with pole on an asymptotically hyperbolic manifold of dimension $n+1\geq 3$, whose metric is at least $C^{2,\alpha}$-conformally compact and polyhomogeneous. 
In Section \ref{effmon}, we prove the monotonicity result, Theorem \ref{thm_monotonicity}. Finally, in Section \ref{sectpositive mass theorem}, we establish our positive mass theorem, Theorem \ref{thm:PMT}, by combining a density argument with the results of Sections \ref{SectEst}. and \ref{effmon}.

\subsection*{Acknowledgements} We appreciate financial support from the G\"{o}ran Gustafsson Foundation. 
The first author would like to thank Hartmut Wei\ss{} for helpful discussions related to this paper.
The second author is a member of the INDAM-GNAMPA. The third author is supported by the FWO and the FNRS via EOS project 40007524.
He would like to express its gratitude to Marco Usula for illuminating discussions.


\section{Preliminaries}\label{SectPreliminaries}

This section is dedicated to collect notations and conventions and to introduce the main objects and properties that will be used throughout the article.

\subsection{Notations and Conventions}
If $(M,g)$ is a Riemannian manifold, its Levi-Civita connection is denoted by $\nabla$, and the associated Laplace-Beltrami is $\Delta = \mathrm{tr}_g(\nabla \circ d)$ (notice that with our convention, $\Delta$ has nonpositive spectrum). Moreover,

\begin{itemize}
\item $\mathcal{H}^k$ is the $k$-dimensional Hausdorff measure induced by the Riemannian distance; 
\item $\mu$ is the canonical measure. 
\end{itemize}

In Sections \ref{effmon} and \ref{sectpositive mass theorem}, we will work only with three-dimensional Riemannian manifolds. In particular, we set:
\begin{itemize}
\item $u=1-4\pi \mathcal{G}_o;$
\item $M_o=M\setminus \{o\};$
\item $\mathcal{T}=\big\{t\in (0,+\infty)\,:\,  \text{$2- \coth t$ is a regular value of $u$}\big\}\,;$
\item $\Sigma_t=\{u=2-\coth t\}$;
\item $\Omega_t=\{u<2-\coth t\}$;
\item $E_{s}^{S}=\left\{s<u<S\right\}$;
\item $\mathrm{Crit}(u)=\{\vert \nabla u \vert=0\}$;
\item $g_{hyp}$ is the canonical metric of the hyperbolic space $\Hyp^3$, and we will explicitly write the subscript $g_{hyp}$ when a quantity is referred to $g_{hyp}$;
\item $b$ is the metric $\varphi^{*}g_{hyp}$, where $\varphi$ is an asymptotically hyperboloidal map of order $\delta>1$, and we will explicitly write the subscript $b$ for any quantity referred to the metric $\varphi^{*}g_{hyp}$;
\item $\mathcal{G}^b=(4\pi)^{-1}\big(\coth r-1\big);$
\item $S_t=\{\varphi_*u=2-\coth t\}$, where $\varphi$ is an asymptotically hyperboloidal map of order $\delta>1$;
\item $D_t$ is the compact domain in $\HH^3$ having $S_t$ as a boundary
\end{itemize}

\subsection{Conformally compact and asymptotically hyperbolic manifolds}

Let $\overline{M}$ be a compact manifold with interior $M$ and boundary $\partial M$.
A \emph{boundary defining function} is a smooth function $\rho \colon \overline{M} \to [0,\infty)$ such that $\rho^{-1}(0) = \partial M$ and $d\rho \neq 0$ along $\partial M$.
Two such functions differ by a multiplicative function that does not vanish along $\partial M$.
A Riemannian metric $g$ on the interior $M$ is called \emph{conformally compact} of class $C^{k,\alpha}$ if $\bar{g} = \rho^2 g$ extends as a $C^{k,\alpha}$ Riemannian metric on $\overline{M}$.
The compact conformal manifold $(\partial M, [\bar{g}|_{\partial M}])$ is called the \emph{conformal infinity} of $(M,g)$.

If $k\geq 2$, $(M,g)$ is a complete noncompact Riemannian manifold whose sectional curvature satisfies
\begin{equation}
    \label{eq:sec_conformal_compactification}
    \sec = - |d\rho|^2_{\bar{g}} + O(\rho).
\end{equation}
Notice that $|d\rho|^2_{\bar{g}}$ does not depend on the choice of $\rho$.
An \emph{asymptotically hyperbolic manifold} is then defined as a conformally compact manifold satisfying $|d\rho|_{\bar{g}} = 1$ along $\partial M$.

Given a representative $h_0 \in [\bar{g}|_{\partial M}])$, it follows from \cite{graham_einstein_1991} that there exists a special geodesic boundary defining function $\rho$, called \emph{geodesic}, such that $|d\rho|_{\bar{g}} =1$ in a neighborhood of $\partial M$, and such that $g$ reads
\begin{equation}
    \label{eq:h_rho}
    g = \frac{d\rho\otimes d\rho + h_{\rho}}{\rho^2},
\end{equation}
in a collar neighborhood of the boundary, identified with $[0,\varepsilon)\times \partial M$ \emph{via} the flow of $\partial_{\rho}$.
Here, $(h_{\rho})_{\rho \in (0,\varepsilon)}$ is a smooth family of Riemannian metrics on $\partial M$ that converges to $h_0$ as $\rho \to 0$ in $C^{k,\alpha}$-topology.

\subsection{Weighted Hölder spaces}

We now introduce some functional spaces that play an important role in the study of conformally compact manifolds.
We refer to \cite{Lee06} for a complete introduction.
Let $(M,g)$ be a conformally compact manifold with boundary defining function $\rho$.
If $C^{\ell,\beta}(M)$ denotes the usual Hölder space on $(M,g)$ of regularity $(\ell, \beta)$, with $\ell$ a nonnegative integer and $\beta \in [0,1]$, then for $\delta \in \mathbb{R}$, we define the weighted Hölder space
\begin{equation}
    C^{\ell,\beta}_{\delta}(M) = \rho^{\delta}C^{\ell,\beta}(M) = \{f \,:\, \exists u \in C^{\ell,\beta}(M) \text{ such that } f = \rho^{\delta}u \} = \{f \,:\, \rho^{-\delta}f \in C^{\ell,\beta}(M)\},
\end{equation}
which is a Banach space when equipped with the weighted norm
\begin{equation}
    \|f\|_{C^{\ell,\beta}_{\delta}(M)} = \|\rho^{-\delta}f\|_{C^{\ell,\beta}(M)}.
\end{equation}
Notice that if $f \in C^{k,\alpha}_{\delta}(M)$, then $|\nabla^j f| = O(\rho^{\delta})$ for all $j \in \{0,\ldots,k\}$.
These spaces do not depend on the choice of the boundary defining function $\rho$ in the sense that they consist of exactly the same functions, and that the associated norms are equivalent.

\subsection{Polyhomogeneity}\label{subsectPolyhomogeneity}

A special type of regularity on conformally compact manifolds is \emph{polyhomogeneity}, which belongs to the more general context of $0$-calculus and $b$-calculus developed by Melrose and Mazzeo \cite{mazzeo_elliptic_1991,mazzeo_elliptic_2014,melrose_transformation_1981}.
See also the recent paper \cite{usula_boundary_2024}, and \cite[Appendix A]{AILA18} for the special case of asymptotically hyperbolic manifolds, which is of interest to us.

Let $\overline{M}$ be a smooth manifold with boundary $\partial M$, interior $M$, and boundary defining function $\rho$.
We denote $\mathcal{A}(M)$ the set of smooth functions $f$ on $M$ such that that $X_1\cdots X_k f$ is bounded for any vector fields $X_1,\ldots,X_k$ on $\overline{M}$ with $X_j|_{\partial M}$ tangent to $\partial M$.
A smooth function $f\colon M \to \mathbb{C}$ is \emph{polyhomogeneous} if
\begin{itemize}
    \item there exists a sequence of complex numbers $\{s_i\}_{i\geq 0}$, with $\mathrm{Re}(s_i)\nearrow \infty$,
    \item there exists a sequence of nonnegative integers $\{N_i\}_{i\geq 0}$,
    \item there exists a sequence of smooth functions $\{\bar{f}_{i,j}\}_{i\geq 0, 0\leq j\leq N_i}$ on $\partial M$,
\end{itemize}
such that for any $k \geq 0$, there exists $N \geq 0$, with
\begin{equation}
    f - \sum_{i=0}^N \sum_{j=0}^{N_i} \bar{f}_{i,j} \rho^{s_i} (\log \rho)^j \in \rho^k \mathcal{A}(M).
\end{equation}
In that case, we write
\begin{equation}
    f \sim \sum_{i\geq 0} \sum_{j=0}^{N_i} \bar{f}_{i,j} \rho^{s_i} (\log \rho)^j.
\end{equation}
An important property of polyhomogeneous functions is that for any $\delta$, there exists $\varepsilon >0$ such that one may write
\begin{equation}
    f = \sum_{\mathrm{Re}(s_i) \leq \delta} \sum_{j=0}^{N_i} \bar{f}_{i,j} \rho^{s_i}(\log \rho)^j + O_{\infty}(\rho^{\delta + \varepsilon}).
\end{equation}

Finally, a conformally compact metric $g$ on $M$ is polyhomogeneous if for any coframe $\{\theta^1,\ldots,\theta^n\}$ on $\partial M$, the components of $g$ in the coframe $\{\rho^{-1}d\rho, \rho^{-1}\theta^1,\ldots,\rho^{-1}\theta^n\}$ are polyhomogeneous.
Notice that if $(M,g)$ is conformally compact of class $C^{\infty}$, then $g$ is polyhomogenous: this can be shown by writing $g$ in the form \eqref{eq:h_rho}, and analyzing the Taylor expansion of $h_{\rho}$ near $\rho = 0$.
In particular, conformally compact polyhomogenous metrics form a dense subset of $C^{k,\alpha}$ conformally compact metrics.

\section[Estimates for the Green function of an AH manifold]{Asymptotic estimates for the Green's function of an asymptotically hyperbolic manifold}\label{SectEst}

Let $(M^{n+1},g)$ be an asymptotically hyperbolic manifold of class $C^{k,\alpha}$, with $n+1 \geq 3$, $k\geq 2$, $\alpha \in (0,1]$.
In this section, we give a proof of the existence and uniqueness of the minimal positive Green function with prescribed pole $o \in M$, and investigate its asymptotic properties, both near the pole and near infinity.
This proves in particular that such manifolds are non-parabolic.
More precisely, the main result of this section is the following.

\begin{theorem}
    \label{thm:asymptotic_phg}
    Let $(M^{n+1},g)$ be an asymptotically hyperbolic manifold of class $C^{k,\alpha}$, with $n+1 \geq 3$, $k\geq 2$, $\alpha \in (0,1]$, and with boundary defining function $\rho$.
    Fix $o \in M$, and consider the differential system
    \begin{equation}
        \label{eq:Green_system}
        \begin{cases}
            \Delta \mathcal{G}_o  = -\delta_o & \text{in } M,       \\
            \mathcal{G}_o > 0                 & \text{in } M,       \\
            \mathcal{G}_o \to 0               & \text{at infinity}.
        \end{cases}
    \end{equation}
    Then there exists a unique solution $\mathcal{G}_o$ to \eqref{eq:Green_system} in the sense of distributions.
    It is a smooth function $ \mathcal{G}_o \colon M\setminus \{o\} \to \mathbb{R}$, and has the following asymptotic properties.
    \begin{enumerate}[label=\arabic*)]
        \item \emph{Asymptotics near the pole.}
              \label{part:near_pole}
              If $\omega_n$ denotes the volume of the $n$-dimensional unit round sphere, and $r = d_g(\cdot,o)$ the geodesic distance to the pole, then
              \begin{align}
                  \label{eq:asymptotic_pole_order_0}
                  \mathcal{G}_o         & \underset{r \to 0}{\sim} \frac{1}{(n-1)\omega_n r^{n-1}},                                             \\
                  \label{eq:asymptotic_pole_order_1}
                  \nabla\mathcal{G}_o   & \underset{r \to 0}{\sim} - \frac{1}{\omega_n r^n} \nabla r,                                           \\
                  \label{eq:asymptotic_pole_order_2}
                  \nabla d\mathcal{G}_o & \underset{r \to 0}{\sim}  \frac{n}{\omega_n r^{n+1}} dr\otimes dr - \frac{1}{\omega_n r^n} \nabla dr.
              \end{align}
        \item \emph{Asymptotics near infinity.}
              \label{part:near_infinity}
              There exists $C >0$ such that, outside some compact region countaining $o$, it holds that
              \begin{align}
                  \label{eq:asymptotic_infinity_order_0}
                  C^{-1} \rho^n \leq \mathcal{G}_o & \leq C \rho^n, \\
                  \label{eq:asymptotic_infinity_order_1}
                  |\nabla \mathcal{G}_o |          & \leq C \rho^n, \\
                  \label{eq:asymptotic_infinity_order_2}
                  |\nabla^2 \mathcal{G}_o |        & \leq C \rho^n.
              \end{align}
        \item \emph{Better asymptotics near infinity in the polyhomogeneous case.}
              \label{part:phg}
              If $g$ is furthermore assumed polyhomogeneous, then there exist two smooth functions $\bar{v}_n, \bar{v}_{n+1} \colon \partial M \to \mathbb{R}$, with $\bar{v}_n >0$, and a constant $\varepsilon >0$, such that
              \begin{equation}
                  \label{eq:asymptotic_phg}
                  \mathcal{G}_o \underset{\rho \to 0}{=} \bar{v}_n \, \rho^n + \bar{v}_{n+1} \, \rho^{n+1} + O_{\infty}(\rho^{n+1+\varepsilon}).
              \end{equation}
    \end{enumerate}
\end{theorem}

The proof is devided in several parts for a better exposition.
Existence and uniqueness of the Green function together with part \ref{part:near_pole} are addressed in Proposition \ref{prop:Green_existence}.
Part \ref{part:near_infinity} is proved in Proposition \ref{prop:coarse_estimates}.
The last part \ref{part:phg} is the subject of Proposition \ref{prop:phg}.

In \eqref{eq:asymptotic_phg}, a neighborhood of the boundary at infinity is identified with a direct product $(0,\varepsilon] \times \partial M$.
We choose $\rho$ to be geodesic, and we henceforth make this assumption throughout all this section.
The metric $g$ then takes the form \eqref{eq:h_rho}, and if $\Delta$, $\bar{\Delta}$, and $\Delta_{h_{\rho}}$, denote the Laplace-Beltrami operators associated with $g$, $\bar{g}$, and $h_{\rho}$, then
\begin{equation}
    \bar{\Delta}
    = \mathrm{tr}_{\bar{g}}({\nabla}^{\bar{g}}\circ d)
    = {\partial_{\rho}}^2 + \frac{1}{2}\mathrm{tr}_{\bar{g}}(\partial_{\rho}h_{\rho}) \partial_{\rho} + \Delta_{h_{\rho}},
\end{equation}
and the conformal rule for the Laplace-Beltrami operator yields
\begin{equation}
    \Delta
    = \rho^2 \bar{\Delta} - (n-1)\nabla \log \rho
    = \rho^2{\partial_{\rho}}^2 - (n-1)\rho\partial_{\rho} + \frac{\rho}{2}\mathrm{tr}_{\bar{g}}(\partial_{\rho}h_{\rho}) \rho\partial_{\rho}+ \rho^2\Delta_{h_{\rho}},
\end{equation}
where we have used that $\nabla \log \rho = \frac{1}{\rho}\nabla\rho = \rho\partial_{\rho}$.
This finally reads
\begin{equation}
    \label{eq:Delta_rho}
    \Delta = (\rho\partial_{\rho})^2 - n \rho\partial_{\rho} + \rho L + \rho^2\Delta_{h_{\rho}}
\end{equation}
with $L$ satisfying
\begin{equation}
    \label{eq:L}
    L
    = \frac{1}{2}\mathrm{tr}_{\bar{g}}(\partial_{\rho}h_{\rho}) \rho\partial_{\rho}
    = (\sigma_0 + O(\rho)) \rho\partial_{\rho},
\end{equation}
for $\sigma_0 = \frac{1}{2}\mathrm{tr}_{\bar{g}}(\partial_{\rho}h_{\rho})|_{\rho=0}$, which is of class $C^{k-1,\alpha}$.

\subsection{Existence and asymptotics near the pole}

We use a perturbative method: we define an approximate solution $G_o$, which we then perturb into a genuine solution $\mathcal{G}_o$ thanks to a result of Lee \cite{Lee06}.
Let $\varepsilon > 0$ be small enough so that $B_o(\varepsilon)$ is a normal neighborhood of $o$ that does not intersect $\rho^{-1}((0,\varepsilon))$.
Let $G_o \colon M\setminus \{o\} \to \mathbb{R}$ be any smooth function satisfying
\begin{equation}
    G_o  = \dfrac{1}{(n-1)\omega_n r^{n-1}}
    \quad \text{in} \quad  B_o(\varepsilon)\setminus\{o\},
    \quad \text{and} \quad
    G_o = \rho^n
    \quad \text{in} \quad  \rho^{-1}((0,\varepsilon)),
\end{equation}
where $r = d_g(\,\cdot\,,o)$ is the geodesic distance to $o$ and $\omega_n$ is the volume of the $n$-dimensional unit sphere.
Classical computations, relying on the fact that $g$ is close to a flat metric in $B_o(\varepsilon)$, show that $-\Delta G_o = \delta_o + w$, with $w$ a smooth function.
Moreover, \eqref{eq:Delta_rho} implies that $w = \rho L(\rho^n) = n/2 \, \partial_{\rho}\mathrm{tr}_{\bar{g}}(\partial_{\rho}h_{\rho})\rho^{n+1}  =  O(\rho^{n+1})$ near infinity.
In particular, $w \in C^{0,\alpha}_{\delta}(M)$ for any $0 < \delta < n$.
In order to perturb $G_o$ into a genuine solution, we will need the following result due to Lee.

\begin{lemma}
    \label{lemma:Lee_isomorphism}
    For $0 < \delta < n$, $\Delta \colon C^{2,\alpha}_{\delta}(M) \to C^{0,\alpha}_{\delta}(M)$ is an isomorphism.
\end{lemma}

\begin{proof}
    This is a direct application of \cite[Theorem C.(c)]{Lee06}.
    Indeed, the indicial radius of $\Delta$ is $n/2$ \cite[Lemma 7.2]{Lee06}, and one only needs to check that $\Delta$ has trivial $L^2$-kernel in $C^{k,\alpha}_{\delta}(M)$.
    But any harmonic function that lives in $L^2(M) \cap C^{k,\alpha}_{\delta}(M)$ vanishes at infinity and therefore must achieve an extremum in $M$.
    Hence, it identically vanishes by the maximum principle.
    This concludes the proof.
\end{proof}

We are now able to prove part \ref{part:near_pole} of Theorem \ref{thm:asymptotic_phg}.

\begin{proposition}
    \label{prop:Green_existence}
    The differential system \eqref{eq:Green_system} admits a unique solution, which satisfies \eqref{eq:asymptotic_pole_order_0}, \eqref{eq:asymptotic_pole_order_1} and \eqref{eq:asymptotic_pole_order_2}.
\end{proposition}

\begin{proof}
    We first prove that a solution to $-\Delta \mathcal{G}_o = \delta_o$ with $\mathcal{G}_o \to 0$ at infinity is unique.
    Indeed, if $\mathcal{G}_o$ and $\mathcal{G}_o'$ are two such solutions, then $\mathcal{G}_o - \mathcal{G}_o'$ is a harmonic function that converges to $0$ at infinity, and must vanish identically by the maximum principle.

    Let us now show the existence of $\mathcal{G}_o$.
    Let $\delta \in (0,n)$ be fixed, and recall that $-\Delta G_o = \delta_o + w$ with $w \colon M \to \mathbb{R}$ a smooth function with $w = O(\rho^{n+1})$.
    In particular, $w \in C^{0,\alpha}_{\delta}(M)$, and Lemma \ref{lemma:Lee_isomorphism} yields the existence of a unique $f\in C^{2,\alpha}_{\delta}(M)$ such that $\Delta f = w$.
    Notice that $f$ is smooth since $w$ is.
    By construction, $\mathcal{G}_o := G_o + f$ solves $\Delta \mathcal{G}_o = - \delta_o$ on $M$, and satisfies $\mathcal{G}_o \to +\infty$ at $o$ and $\mathcal{G}_o \to 0$ at infinity.
    The maximum principle then ensures that $\mathcal{G}_o >0$ on $M\setminus \{o\}$.
    Consequently, $\mathcal{G}_o$ solves \eqref{eq:Green_system}.

    To conclude, one only needs to show that \eqref{eq:asymptotic_pole_order_0}, \eqref{eq:asymptotic_pole_order_1}, and \eqref{eq:asymptotic_pole_order_2} hold.
    This is a straightforward consequence of the fact that $\mathcal{G}_o = G_o + f$ with $f$ a smooth function, and of the explicit form of $G_o = \frac{1}{(n-1)\omega_n r^{n-1}}$ near $o$.
\end{proof}

\subsection{Coarse asymptotics near the boundary}

This subsection is devoted to the proof of part \ref{part:near_infinity} of Theorem \ref{thm:asymptotic_phg}.

\begin{proposition}
    \label{prop:coarse_estimates}
    Let $\mathcal{G}_o$ be the unique solution to \eqref{eq:Green_system}.
    Then there exists $C >0$ such that \eqref{eq:asymptotic_infinity_order_0}, \eqref{eq:asymptotic_infinity_order_1}, and \eqref{eq:asymptotic_infinity_order_2} hold outside some compact region containing $o$.
\end{proposition}

\begin{proof}
    Fix $\varepsilon \in (0,1) $ small enough so that $(0,\varepsilon]$ does not contain any critical value of $\rho$.
    Consider the complete manifold with boundary $N = \rho^{-1}((0,\varepsilon]) \simeq (0,\varepsilon] \times \partial M$, and define two functions $G_{\pm}$ by
    \begin{equation}
        \forall (\rho,y)\in N \simeq (0,\varepsilon] \times \partial M, \quad
        G_{\pm}(\rho,y) = \mathcal{G}_o(\varepsilon,y)\frac{\rho^n \mp \rho^{n+1/2}}{\varepsilon^n \mp \varepsilon^{n+1/2}}.
    \end{equation}
    Using \eqref{eq:Delta_rho}, one readily checks that
    \begin{equation}
        \begin{split}
            \Delta G_{\pm}(\rho,y)
            &= \mathcal{G}_o(\varepsilon,y) \frac{\mp\frac{2n+1}{4}\rho^{n+1/2} + \rho L(\rho^n \mp \rho^{n+1/2})}{\varepsilon^n \mp \varepsilon^{n+1/2}}
            + \rho^2 \Delta_{h_{\rho}}\mathcal{G}_o(\varepsilon,y) \frac{\rho^n \mp \rho^{n+1/2}}{\varepsilon^n \mp \varepsilon^{n+1/2}}\\
            &= \mp \frac{(2n+1)\mathcal{G}_o(\varepsilon,y)}{4(\varepsilon^n \mp \varepsilon^{n+1/2})} \rho^{n+1/2} + O(\rho^{n+1}).
        \end{split}
    \end{equation}
    Here, we have used that $L(\rho^{\alpha}) = O(\rho^{\alpha})$ for any $\alpha$, and $\Delta_{h_{\rho}}G(\varepsilon,y) = O(1)$ as $\rho \to 0$ since $h_{\rho} \to h_0$ in $C^{k,\alpha}$-topology.
    In particular, if $\varepsilon$ is chosen small enough, $G_+$ is superharmonic, and $G_-$ is subharmonic.
    By construction, they coincide with $\mathcal{G}_o$ on $\partial N$, and at infinity.
    It now follows from the maximum principle that $G_- \leq \mathcal{G}_o \leq G_+$ on $N$, which implies as a byproduct that
    \begin{equation}
        \frac{\min\{\mathcal{G}_o(x) \mid x \in \rho^{-1}(\varepsilon)\}}{\varepsilon^n + \varepsilon^{n+1/2}} \rho^n
        \leq \mathcal{G}_o
        \leq \frac{\max\{\mathcal{G}_o(x) \mid x \in \rho^{-1}(\varepsilon)\}}{\varepsilon^n - \varepsilon^{n+1/2}} \rho^n
        \quad \text{in} \quad \rho^{-1}((0,\varepsilon]),
    \end{equation}
    and \eqref{eq:asymptotic_infinity_order_0} now follows by choosing $C$ large enough.

    It remains to show that \eqref{eq:asymptotic_infinity_order_1} and \eqref{eq:asymptotic_infinity_order_2} hold.
    Consider a smooth function $\chi$  on $M$ satisfying $\chi \equiv 0$ near $o$, and $\chi \equiv 1$ near infinity.
    Then $\chi \mathcal{G}_o \in C^{0,0}_n(M)$ by the first part of the proof.
    In addition, $\Delta(\chi\mathcal{G}_o)$ is smooth and compactly supported in $M$, so that $\Delta(\chi\mathcal{G}_o)\in C^{k-2,\alpha}_n(M)$.
    It then follows from \cite[Lemma 4.8(b)]{Lee06} that there exists a constant $C >0$ such that
    \begin{equation}
        \|\chi\mathcal{G}_o\|_{C^{k,\alpha}_n(M)} \leqslant C\left(\|\Delta (\chi\mathcal{G}_o)\|_{C^{k-2,\alpha}_n(M)} + \|\chi\mathcal{G}_o\|_{C^{0,0}_n(M)}\right) < +\infty.
    \end{equation}
    In other words, $\chi\mathcal{G}_o \in C^{k,\alpha}_n(M)$ with $k \geqslant 2$.
    In particular, $\nabla \mathcal{G}_o, \nabla^2 \mathcal{G}_o = O(\rho^n)$ as $\rho \to 0$, which concludes the proof.
\end{proof}

\subsection{Precise asymptotics near infinity for polyhomogeneous metrics}
We now aim at proving part \ref{part:phg} of Theorem \ref{thm:asymptotic_phg}.
We stress that we now assume $g$ to be polyhomogeneous in a neighborhood of the  boundary at infinity.

\begin{proposition}
    \label{prop:phg}
    If $g$ is assumed polyhomogeneous at infinity, then there exist two smooth functions $\bar{v}_n,\bar{v}_{n+1} \colon \partial M \to \mathbb{R}$, with $\bar{v}_n >0$, and $\varepsilon >0$, such that \eqref{eq:asymptotic_phg} holds.
\end{proposition}

\begin{proof}
    The proof goes as follows.
    Set $v = \chi \mathcal{G}_o$, where $\chi \colon M \to [0,1]$ is smooth, identically vanishes near $o$, and is constant equal to $1$ outside some compact set.
    By construction, $v$ is a smooth function on $M$, and coincides with $\mathcal{G}_o$ near infinity.
    One thus only needs to show that \eqref{eq:asymptotic_phg} holds for $v$.
    We first apply a result of \cite{AILA18} to obtain polyhomogeneity for $v$.
    Then, we exploit the expression \eqref{eq:Delta_rho} to show that logarithmic terms in the polyhomogeneous expansion of $v$ cannot appear at order less or equal than $n+1$.

    The indicial radius of the Laplace-Beltrami operator of an asymptotically hyperbolic manifold of dimension $n+1$ acting on functions is $n/2$ \cite[Lemma 7.2]{Lee06}.
    Hence, \cite[Theorem A.14]{AILA18} states that if $f$ is polyhomogeneous and $0<\delta<n$, any solution of $\Delta u = f$ with $u \in C^{2,\alpha}_{\delta}(M)$ is polyhomogeneous.
    Now, recall that $\Delta v$ has compact support in $M$, and is therefore trivially polyhomogeneous.
    Since $v \in C^{2,\alpha}_n(M)$ (see Proposition \ref{prop:coarse_estimates}), then $v \in C^{2,\alpha}_{\delta}(M)$ for any $0<\delta<n$.
    It thus follows that $v$ is polyhomogeneous:
    there exist a sequence of complex numbers $\{s_i\}_{i\geq 0}$ with $\mathrm{Re}(s_i) \nearrow \infty$, a sequence of nonnegative integers $\{N_i\}_{i\geq 0}$, and a sequence of smooth complex valued functions $\{\bar{v}_{i,j}\}_{i \geq 0, 0\leq j \leq N_i}$ on $\partial M$, such that
    \begin{equation}
        v \sim \sum_{i\geq 0} \sum_{j=0}^{N_i} \bar{v}_{i,j} \, \rho^{s_i} (\log\rho) ^j.
    \end{equation}
    We first remark that $v = O(\rho^n)$ by \eqref{eq:asymptotic_infinity_order_0}, so that all coefficients $\bar{v}_{i,j}$ with $\mathrm{Re}(s_i) \leq n$ vanish unless $\mathrm{Re}(s_i) = n$ and $j=0$.
    Let us show that actually, these coefficients vanish unless $(s_i,j) = (n,0$).
    To that end, let us write $v = v_n + v_{+}$, where
    \begin{equation}
        v_n = \sum_{\mathrm{Re}(s_i) = n} \bar{v}_{i,0} \rho^{s_i}
        \quad \text{and} \quad
        v_+ \sim \sum_{\mathrm{Re}(s_i) > n} \sum_{j=0}^{N_i} \bar{v}_{i,j} \rho^{s_i}  (\log\rho)^j.
    \end{equation}
    Then there exists $\varepsilon_+ >0$ such that $v_+,\Delta v_+ = O(\rho^{n+\varepsilon_+})$, and equations \eqref{eq:Delta_rho} and \eqref{eq:L} yield
    \begin{equation}
        \begin{split}
            \Delta v & = \Delta v_n + \Delta v_+ \\
            & = \sum_{\mathrm{Re}(s_i)=n}((\rho\partial_{\rho})^2 - n\rho\partial_{\rho} + \rho L + \rho^2\Delta_{h_{\rho}})(\bar{v}_{i,0}\rho^{s_i}) + \Delta v_+ \\
            & = \sum_{\mathrm{Re}(s_i)=n}\left(\bar{v}_{i,0}((\rho\partial_{\rho})^2 - n\rho\partial_{\rho})(\rho^{s_i}) + \bar{v}_{s,0} \rho L (\rho^{s_i}) + (\Delta_{h_{\rho}}\bar{v}_{i,0}) \rho^{s_i+2}\right) + \Delta v_+ \\
            & = \sum_{\mathrm{Re}(s_i)=n}\bar{v}_{i,0} s_i(s_i-n) \rho^{s_i} + o(\rho^n).
        \end{split}
    \end{equation}
    Since $\Delta v$ identically vanishes near infinity, it follows that $\bar{v}_{i,0}=0$ whenever $\mathrm{Re}(s_i) = n$, $s_i \neq n$.
    Let us then write $\bar{v}_n$ for the coefficient corresponding to $s_i = n$, and, up to relabeling, we may now write $v = \bar{v}_n \rho^n + v' + v''$, with
    \begin{equation}
        v' = \sum_{i=1}^k \sum_{j=0}^{N_i} \bar{v}_{i,j} \rho^{s_i} (\log\rho)^j
        \quad \text{and} \quad
        v'' \sim \sum_{i > k} \sum_{j=0}^{N_i} \bar{v}_{i,j} \rho^{s_i} (\log\rho)^j
    \end{equation}
    satisfying $n < \mathrm{Re}(s_i) \leq n+1$ for $1 \leq i \leq k$, and $v'' = O(\rho^{n+1+\varepsilon})$ for some $\varepsilon > 0$.
    Then $\Delta v'' = O(\rho^{n+1+\varepsilon})$ and using \eqref{eq:Delta_rho} yields
    \begin{equation}
        \begin{split}
            \Delta v
            & = \Delta(\bar{v}_n\rho^n) + \Delta v' + \Delta v'' \\
            & =  ((\rho\partial_{\rho})^2 - n\rho\partial_{\rho} + \rho L + \rho^2\Delta_{h_{\rho}})(\bar{v}_n\rho^n) \\
            & \quad + \sum_{i=1}^k \sum_{j=0}^{N_i} ((\rho\partial_{\rho})^2 - n\rho\partial_{\rho} + \rho L + \rho^2\Delta_{h_{\rho}})(\bar{v}_{i,j}\rho^{s_i}(\log\rho)^j) + \Delta v''. \\
        \end{split}
    \end{equation}
    Notice first that \eqref{eq:L} yields
    \begin{equation}
        ((\rho\partial_{\rho})^2 - n\rho\partial_{\rho} + \rho L + \rho^2\Delta_{h_{\rho}})(\bar{v}_n\rho^n)
        = n\sigma_0\bar{v}_n\rho^{n+1} + O(\rho^{n+2}).
    \end{equation}
    Similarly, if $1 \leq i \leq k$,
    \begin{equation}
        ((\rho\partial_{\rho})^2 - n\rho\partial_{\rho} + \rho L + \rho^2\Delta_{h_{\rho}})(\bar{v}_{i,j}\rho^{s_i}(\log\rho)^j)
        = \bar{v}_{i,j}((\rho\partial_{\rho})^2 - n\rho\partial_{\rho})(\rho^{s_i}(\log\rho)^j) + o(\rho^{n+1}).
    \end{equation}
    It now follows that
    \begin{equation}
        \Delta v = n\sigma_0 \bar{v}_n \rho^{n+1} + \sum_{i=1}^k \sum_{j=0}^{N_i} \bar{v}_{i,j}((\rho\partial_{\rho})^2 - n\rho\partial_{\rho})(\rho^{s_i}(\log\rho)^j) + o(\rho^{n+1}),
    \end{equation}
    with terms $((\rho\partial_{\rho})^2 - n\rho\partial_{\rho})(\rho^{s_i}(\log\rho)^j)$ that are not of order $o(\rho^{n+1})$.
    Since $\Delta v$ is polyhomogeneous with identically vanishing polyhomogeneous expansion, it is necessary that
    \begin{equation}
        \label{eq:coeff_phg}
        n\sigma_0 \bar{v}_n \rho^{n+1} + \sum_{i=1}^k \sum_{j=0}^{N_i} \bar{v}_{i,j}((\rho\partial_{\rho})^2 - n\rho\partial_{\rho})(\rho^{s_i}(\log\rho)^j) = 0.
    \end{equation}
    Fix $ 1 \leq i \leq k$ such that $s_i \neq n+1$.
    The coefficient of order $(s_i,N_i)$ in \eqref{eq:coeff_phg}, which has to vanish, is $\bar{v}_{i,N_i}s_i(s_i-n)$.
    Hence, it is necessary that $\bar{v}_{i,N_i}= 0$.
    A finite induction process then shows that $\bar{v}_{i,j} = 0$ for all $0 \leq j \leq N_i$.
    In addition, an identical argument shows that $\bar{v}_{i,j} = 0$ whenever $s_i = n+1$ and $j \geq 1$, and that
    \begin{equation}
        n\sigma_0 \bar{v}_n + (n+1) \bar{v}_{n+1}= 0,
    \end{equation}
    where $\bar{v}_{n+1}$ denotes the coefficient $\bar{v}_{i,0}$ such that $s_i = n+1$.
    This proves that $v' = \bar{v}_{n+1}\rho^{n+1}$.
    We have then shown that there exist two smooth functions $\bar{v}_n,\bar{v}_{n+1} \colon \partial M \to \mathbb{C}$, and $\varepsilon > 0$, such that
    \begin{equation}
        \label{eq:v_expansion}
        v = \bar{v}_n \rho^n + \bar{v}_{n+1} \rho^{n+1} + O(\rho^{n+1+\varepsilon}).
    \end{equation}
    Since $v$ is real valued, \eqref{eq:v_expansion} immediatly implies that $\bar{v}_n$ and $\bar{v}_{n+1}$ are real valued as well.
    Moreover, $\bar{v}_n >0$ by virtue of \eqref{eq:asymptotic_infinity_order_0}.

    Finally, one can differentiate \eqref{eq:v_expansion} term by term in $M$ by polyhomogeneity of $v$.
    This concludes the proof.
\end{proof}

\begin{corollary}
    \label{cor:no_critical_points}
    Under the assumptions of Theorem \ref{thm:asymptotic_phg} \ref{part:phg}, $\mathcal{G}_o$ has no critical point in a neighborhood of the boundary at infinity.
\end{corollary}

\begin{proof}
    Differentiating \eqref{eq:asymptotic_phg} yields $d \mathcal{G}_o(\partial_{\rho}) =n \bar{v}_n \rho^{n-1} + O(\rho^n)$.
    To conclude, recall that $\bar{v}_n >0$, so that $d\mathcal{G}_o$ does not vanish in a neighborhood of the boundary.
\end{proof}

\section{A monotonicity formula for the Green functions}\label{effmon}
 
This section is devoted to the proof of Theorem \eqref{thm_monotonicity}. Before seeing the proof, we need to observe that the function $F$ is well-defined. We start noticing that $u$ is smooth on $M \setminus \{ o\}$ and $u$ assumes values in the interval $(-\infty,1)$, by the maximum principle. 
Moreover, the level sets of $u$ are compact, while the sub-level sets are contained in compact subsets of $M$. 
Therefore, by~\cite[Theorem~1.7]{Hardt1}, every level set have finite $2$-dimensional Hausdorff measure. Now, recalling that the set $\mathrm{Crit}(u)=\{\vert \nabla u \vert=0\}$ has locally finite $1$-dimensional Hausdorff measure (see for instance~\cite[Theorem~1.1]{Hardt2}), it is not difficult to see that the first three terms of the function $F$, given by \eqref{eq0}, are well-defined. Indeed, these terms are obtained by integrating
$\mathcal{H}^2$-almost everywhere defined and bounded functions on sets with finite measure, keeping in mind that, since $\mathrm{H}$ can be expressed as 
\begin{equation}\label{H}
\mathrm{H} \, = \, - \, \frac{\nabla du\, (\nabla u,\nabla u)}{\vert \nabla u\vert^{3}}\,=\,-\,\frac{\langle\nabla \vert \nabla u\vert, \nabla u\rangle}{ \vert \nabla u\vert^{2}}
\end{equation}
away from $\mathrm{Crit}(u)$, as $u$ is harmonic, it turns out that
\begin{equation}
\label{eq:hdubound}
\big\vert \,\vert \nabla u \vert \mathrm{H}\,\big\vert\leq \vert\nabla d u (\nu,\nu) \vert\leq \vert \nabla du \vert  \, .
\end{equation}
Finally, by the asymptotic behavior of $\mathcal{G}_o$ near the pole (see, for instance, \cite[Appendix]{MarRigSet}), in a sufficiently small punctured neighborhood of $o$, we have
\begin{equation}
\frac{C_1}r\leq 1-u\leq\frac{C_2}r,
\qquad
\frac{C_3}{r^2}\leq\vert \nabla u \vert \leq \frac{C_4}{r^2},
\end{equation}
for some positive constants $C_{i}>0$, $i=1, \ldots, 4$. Then, combining these bounds yields
\begin{equation}
\frac{\vert \nabla u \vert}{(2-u)^2-1} \leq \frac{C_4}{r^2} \frac{1}{(1+\frac{C_1}r)^2-1}=\frac{C_4}{C_1^2(1+\frac{2r}{C_1})}\leq \frac{C_4}{C_1^2}\,,
\end{equation}
and we can conclude that the last two terms of $F$ are well defined, since they also obtained by integrating bounded functions over sets with finite measure.
\smallskip

Before presenting the proof of Theorem \eqref{thm_monotonicity}, we  further observe that, for every value $s\in (-\infty,1)$ of $u$, the open set $\{u<s\}$ is connected, and every connected components of $\{u>s\}$ is unbounded. These facts are consequence of the maximum principles and the size of the set $\mathrm{Crit}(u)$. Indeed, there exists a unique connected component of $\{u<s\}$, such that the pole $o$ is an interior point in its closure. Then, every other connected component has compact closure contained in $M\setminus \{o\}$, and its boundary is contained in $\{u=s\}$. By the maximum principle, it then follows that $u$ is constant therein, but this is no possible since $\mathrm{Crit}(u)\cap \overline{\{u<s\}}$ has finite $1$-dimensional Hausdorff measure.
Similarly, one can prove that there are no bounded connected components of the open $\{u>s\}$.
\smallskip

We are now in a position to present the proof of Theorem \eqref{thm_monotonicity}.

\begin{proof}[Proof of Theorem \ref{thm_monotonicity}]
We start by observing that it is enough to show that the function $F$ admits a locally absolutely
continuous representative in $(0,+\infty)$, which coincides with $F$ on the set 
\begin{equation}\label{defmathcalT}
\mathcal{T}=\big\{t\in (0,+\infty)\,:\,  \text{$2- \coth t$ is a regular value of $u$}\big\}\,.
\end{equation}
We recall that $\mathcal{T}$ is an open set of $(0,\infty)$, by the same argument as in \cite[Theorem 2.3]{AMO}, and its complementary in $(0,+\infty)$ has zero Lebesgue measure, by Sard's theorem.
Now, on the open set $\mathcal{T}$ the function $F$ is continuously differentiable, and from evolution equations for hypersurfaces moving in normal directions (see for example~\cite[Theorem 3.2]{Hui_Pol}) it follows that, for every $t\in\mathcal{T}$,
\begin{align}
\frac{d}{dt} \,\int\limits_{\Sigma_{t}} \vert \nabla u \vert^{2} \, d\mathcal{H}^2 \, &  = \, - \, \frac{1}{\sinh^{2}t} \,\int\limits_{\Sigma_{t}} \vert \nabla u \vert\,\mathrm{H}\,  d\mathcal{H}^2\, ,\nonumber \\
\frac{d}{dt} \,\int\limits_{\Sigma_{t}} \vert \nabla u \vert\, \mathrm{H} \, d\mathcal{H}^2 & \, = \, -\,\frac{1}{\sinh^2 t}\, \int\limits_{\Sigma_{t}} \vert \nabla u \vert  \left[ \, \Delta_{\Sigma_{t}} \!\left(\frac{1}{\vert \nabla u \vert}\right) + \, \frac{\vert \mathrm{h}\vert^{2}+\Ric (\nu,\nu)}{\vert \nabla u\vert} \, \right]\,d\mathcal{H}^2\,,\quad\label{eq20}
\end{align}
where $\Sigma_{t}=\{u=2-\coth t\}$ and $\Delta_{\Sigma_{t}}$ is the Laplace--Beltrami operator of the metric induced on $\Sigma_{t}$.
With the help of the traced Gauss equation, the integrand on the right hand side of~\eqref{eq20} can be expressed as 
\begin{align*}
\vert \nabla u \vert & \left[ \, \Delta_{\Sigma_{t}} \!\left(\frac{1}{\vert \nabla u \vert}\right) + \, \frac{\vert \mathrm{h}\vert^{2}+\Ric (\nu,\nu)}{\vert \nabla u\vert} \, \right]
= &\\
 & \qquad \qquad \qquad \qquad = \,-\, \Delta_{\Sigma_{t}}(\log \vert \nabla u \vert ) \,+\,\frac{\vert\,\nabla^{\Sigma_t}\vert \nabla u\vert\,\vert^{2}}{\vert \nabla u \vert^{2}} \,+\,\frac{\Ro}{2}\, -\,\frac{\,\rm{R}^{\Sigma_{t}}}{2} \,+\, \frac{\,\vert \ringg{\mathrm{h}}\vert^{2}}{2}   \,+\,\frac{3}{4}\,\mathrm{H}^{2}\,.\nonumber
\end{align*}
Substituting the latter expression into~\eqref{eq20} and by the coarea formula, we get 
\begin{align*}
F'(t) \, &= \, 4\pi\,+\,(3\sinh^2 t \cosh^2 t+\sinh^4 t) \int\limits_{\Sigma_{t}}\vert \nabla u \vert^{2} \,\, d\mathcal{H}^2 \, \,-\,\,3\sinh t \cosh t \int\limits_{\Sigma_{t}}\vert \nabla u \vert\, \mathrm{H}   \,\, d\mathcal{H}^2\\
&\quad +\, \int\limits_{\Sigma_{t}} \bigg[\frac{\vert\,\nabla^{\Sigma_t}\vert \nabla u\vert\,\vert^{2}}{\vert \nabla u \vert^{2}} +\frac{\Ro}{2} -\!\frac{\,\rm{R}^{\Sigma_{t}}}{2} + \frac{\,\vert \ringg{\mathrm{h}}\vert^{2}}{2}   +\frac{3}{4}\mathrm{H}^{2}\Bigg]\,\, d\mathcal{H}^2\,\,+\,\,3\, \mathrm{Area}\big(\Sigma_{t}\big)\\
&\quad -\,\,\sinh^4t \int\limits_{\Sigma_{t}}\vert \nabla u \vert^{2} \,\, d\mathcal{H}^2\,,
\end{align*}
for every $t\in\mathcal{T}$, which implies 
\begin{align}
F'(t) &=  4\pi - \int\limits_{\Sigma_t}\frac{\,\rm{R}^{\Sigma_{t}}}{2} \, d\mathcal{H}^2  \,+\,  \int\limits_{\Sigma_t}\frac{\,\rm{R} +6}{2} \, d\mathcal{H}^2\,+  \,\int\limits_{\Sigma_t}\bigg[ \, \frac{\vert\,\nabla^{\Sigma_t}\vert \nabla u\vert\,\vert^{2}}{\vert \nabla u \vert^{2}}+\frac{\,\vert \ringg{\mathrm{h}}\vert^{2}}{2}\,\,\bigg]\, d\mathcal{H}^2\,,\\
&\quad\,+\,\frac{3}{4}\int\limits_{\Sigma_t}\!\bigg(\,\mathrm{H}-\frac{2(2-u)}{(2-u)^2-1} \,\vert \nabla u \vert\bigg)^{\!2}d\mathcal{H}^2\,.\label{monoliscia}
\end{align}
Now, we notice that the last three summands of the right hand side are always non--negative, as the scalar curvature of $(M,g)$ is greater than or equal to $-6$ by assumption. The first two summands also give a non-negative contribution, by virtue of Gauss-Bonnet theorem. Indeed, by \cite[Lemma 2.3]{MW24} and by what was observed immediately before this proof, every regular level set of $u$ is either connected, or, if it is not connected, each connected component is not diffeomorphic to a $2$-sphere. Thus, $F'(t)\geq 0$ for every $t\in\mathcal{T}$, which, together with the fact that $F$ admits a locally absolutely continuous representative in $(0,+\infty)$ coinciding with it on $\mathcal{T}$, yields the desired monotonicity. \\
Let us now prove that $F$ admits a locally absolutely
continuous representative in $(0,+\infty)$, and we immediately observe that this representative coincides with $F$ on the set $\mathcal{T}$, as $\mathcal{T}$ is an open subset of $(0,+\infty)$ and therein the function $F$ is continuous. 
Moreover, from \cite[Lemma 12]{ChLi} it follows that the function
\begin{equation}\label{eq1}
s\in (-\infty,1)\longmapsto \int\limits_{\{u=s\}}\!\!\!\vert \nabla u \vert^{2} \,d\mathcal{H}^2\in (0,\infty)
\end{equation}
is locally Lipschitz, hence, the second term in expression \eqref{eq0} of $F$ determine a locally Lipschtz function in the open interval $(0,+\infty)$. At same time, since the coarea formula 
implies that the functions
\begin{align}
s\in (-\infty,1)&\longmapsto \frac{1}{(2-s)^2-1}\,\mathrm{Area}(\{u=s\})\in (0,\infty)\\
s\in (-\infty,1)&\longmapsto \frac{1}{[(2-s)^2-1]^3} \int\limits_{\{u=s\}} \!\!\!\vert \nabla u \vert^2 \,\,d\mathcal{H}^2\in (0,\infty)
\end{align}
are locally integrable, the functions
\begin{align}
s\in (-\infty,1)\longmapsto \int\limits_{\{u<s\}}\!\!\!\frac{\vert \nabla u \vert}{(2-u)^2-1} \,\, d\mu\quad\quad \text{and}\quad\quad s\in (-\infty,1)\longmapsto\int\limits_{\{u<s\}}\!\!\!\frac{\vert \nabla u \vert^3}{[(2-u)^2-1]^3} \,\, d\mu\,,
\end{align}
are locally absolutely continuous. This implies that the last two terms in expression \eqref{eq0} of $F$ determine functions that are locally absolutely continuous in the open interval $(0,+\infty)$. Then, the statement to be proved follows once we have shown that the auxiliary function
$$\widetilde{F}\,:\,s\in (-\infty,1)\longmapsto -\int\limits_{\{u=s\}}\!\!\!\vert \nabla u \vert\, \mathrm{H}   \,\, d\mathcal{H}^2\in\R$$
belongs $W^{1,1}_{loc}(-\infty,1)$. To show this, let us consider the vector field $Y$, given by
\begin{align}\label{Y}
Y\,=\,\nabla \vert \nabla u\vert\,,
\end{align}
which is well defined and smooth on the open set $M_{o}\setminus \mathrm{Crit}(u)$, where $M_{o}$ is defined as
$$M_{o}=\,M\setminus\{o\} \, .
$$ 
Notice that 
\begin{equation}
\widetilde{F}(s)\,=\int\limits_{\{u=s\}}\! \!\!\bigg\langle Y , \, \frac{\nabla u}{|\nabla u|}\bigg\rangle \,\, d\mathcal{H}^2
\end{equation}
everywhere and the divergence of $Y$ on $M_{o}\setminus \mathrm{Crit}(u)$ can be expressed as 
\begin{equation}
\mathrm{div}(Y)=\,\vert \nabla u \vert^{-1}\Big(
\vert \nabla du\vert^{2}\,-\,\vert\,\nabla\vert \nabla u\vert\,\vert^{2}\,+\,\Ric(\na u,\na u)\Big)\,,
\end{equation}
by the Bochner formula and the fact that $u$ is harmonic. We claim that 
\begin{equation}
\label{eq:claim}
\mathrm{div}(Y)\in L^{1}_{loc}(M_o) \,.
\end{equation}
Indeed, if $K$ is a compact subset of $M_o$, then, by Sard's Theorem, it is contained in the set 
$
E_{s}^{S}=\left\{s<u<S\right\} ,
$
for some regular values $s,S$ of $u$ such that $-\infty<s<S<1$. 
Similarly to the proof of \cite[Theorem 1.1]{AMO24}, let us consider a sequence of cut-off functions $\{ \eta_k\}_{k \in \mathbb{N}^+}$, where, for every $k \in \mathbb{N}^+$, the  function $\eta_k:[0,+\infty) \to [0,1]$ is smooth, non-decreasing, and such that
\begin{align}
\eta_k (\tau) \equiv 0\quad \text{in $\left[0 \,,\frac{1}{2k}\,\right]$}\,,\qquad 
0\leq \eta_k'(\tau)\leq 2 k\quad \text{in $\left[\,\frac{1}{2k}\,,\frac{3}{2k}\,\right]$}\,,
\qquad
\eta_k(\tau) \equiv 1\quad\text{in $\left[\,\frac{3}{2k} \, ,+\infty\!\right)$}\,.\,\nonumber
\end{align}
We use these cut-off functions to define, for every $k \in\mathbb{N}^+$, the vector field
\begin{equation}
Y_{k}\, = \, \, \eta_{{k}}\big( \vert \nabla u\vert\big)\,Y \, ,\,\nonumber
\end{equation}
which is smooth on all $M_{o}$. 
For any such $Y_k$, the divergence is given by
\begin{align}\label{eq2}
\mathrm{div}(Y_{k})
&=\, \eta_{{k}}\big( \vert \nabla u\vert\big)\mathrm{div}(Y )\,+\,\eta_{{k}}'\big( \vert \nabla u\vert\big)\,\vert\,\nabla\vert \nabla u\vert\,\vert^{2}\,,
\end{align}
and, on any compact subset of $M_{o}\setminus \mathrm{Crit}(u)$, it coincides with the vector field $Y$, provided $k$ is large enough.
By these considerations and applying the divergence theorem, it holds that
\begin{equation}
\label{tfci}
\widetilde{F}(S)-\widetilde{F}(s) \,=\, \int\limits_{E_s^S} \!\mathrm{div}(Y_k) \, d\mu \, \geq \, \int\limits_{E_s^S} \! P_k \, d\mu \,+ \,\int\limits_{E_s^S}\!D_k \, d\mu\,,
\end{equation}
where we set 
\begin{align}
P_k & \, = \, \eta_{{k}}\big(\vert \nabla u\vert\big)\,P\,,\quad \text{with}\quad P\,=\,\vert \nabla u \vert^{-1}\Big( \vert \nabla du\vert^{2}-\vert\,\nabla\vert \nabla u\vert\,\vert^{2}\Big) \,,\nonumber\\
D_k & \, =\,  \eta_{{k}}\big(\vert \nabla u\vert \big)\,D\,,\,\quad \text{with}\quad D\,=\,\vert \nabla u \vert^{-1}\,\Ric(\na u,\na u) \, .\nonumber
\end{align}
Now, since the functions $D_k$ satisfy the inequality
$$
\vert D_k \vert
\,\, \leq \,\,\vert \nabla u \vert \, \vert\Ric\vert \,\in L^{1}_{loc}(M_o) \,,
$$
applying Lebesgue's dominated convergence theorem yields $D\in L^{1}(E_{s}^{S})$ and 
\begin{equation*}
\lim_{k\to + \infty} \,\,\int\limits_{E_s^S}\! D_k \, d\mu\,=  \,\int\limits_{E_s^S} \!D \,   d \mu\, < +\infty \,.
\end{equation*}
This fact, combined with inequality~\eqref{tfci}, implies that the sequence of the integrals of the functions $P_k$ is uniformly bounded in $k$. On the other hand, the $P_k$'s are clearly nonnegative, and they converge monotonically and pointwise 
to the function $P$, outside the set of the critical points of $u$. 
Thus, the monotone convergence theorem yields 
\begin{equation*}
\lim_{k\to + \infty} \,\,\int\limits_{E_s^S}\!P_k \, d\mu\, =  \int\limits_{E_s^S} \!P\, d \mu  \, < +\infty\,.
\end{equation*}
In particular, we have that $P\in L^{1}(E_{s}^{S})$.
Since $\mathrm{div}(Y)=P+D$, it follows then that $\mathrm{div}(Y)\in L^{1}_{loc}(M_o)$, as desired.

Having the claim~\eqref{eq:claim} at hand, we are now ready to prove that $\widetilde{F} \in W^{1,1}_{loc}(-\infty,1)$ with weak derivative given by
\begin{align}\label{Phi'}
\widetilde{F}'(s)&\,=\int\limits_{\{u=s\}}\!\!\! \vert \nabla u \vert^{-1}\, \mathrm{div}(Y)\,d\mathcal{H}^2
\end{align}
a.e. in $(-\infty,1)$. First, we observe that the latter belongs to $L^1_{loc}(-\infty,1)$, thanks to the coarea formula, along with the facts that $ \vert \nabla u \vert^{-1}\,|D|\leq \vert\Ric\vert \in L^{1}_{loc}(M_o)$ and $P\in L^{1}_{loc}(M_o)$. 
Let us consider a test function $\chi\in C_{c}^{\infty}(-\infty,1)$. We have 
\begin{align}
\int\limits_{-\infty}^{1}\!\!\chi'(\tau)\,\widetilde{F}(\tau)\,d\tau&\,=\,\int\limits_{-\infty}^{1}d\tau\int\limits_{\{u=s\}}\!\!\!\chi'(u)\,\,\bigg\langle Y , \, \frac{\nabla u}{|\nabla u|}\bigg\rangle \,d\mathcal{H}^2\,=\,\int\limits_{M_{o}}\! \Big\langle Y,\,\nabla \chi (u)\Big\rangle\,d\mu\\
&\,=\,\lim_{k\to +\infty}\,\int\limits_{M_{o}}\!\Big\langle Y_k,\,\nabla \chi (u)\Big\rangle\,d\mu\,=\,-\,\lim_{k\to +\infty}\,\int\limits_{M_{o}}\!\chi(u)\,\mathrm{div}( Y_{k})\,d\mu\nonumber\,,
\end{align}
where the second equality follows by the coarea formula, the third one by Lebesgue's dominated convergence theorem, whereas the last one is a simple integration by parts. Let us put $-\infty<s<S<1$ such that $\mathrm{supp}\chi \subset (s,S)$ and $s,S$ are regular values of $u$. 
By identity~\eqref{eq2}, it follows that
\begin{align}
\int\limits_{M_{o}}\!\chi(u)\,\mathrm{div}( Y_{k})\,d\mu&\,=\int\limits_{E_{s}^{S}}\!\chi(u)\,\Big[P_{k}\,+\,D_{k}\,+\,\eta_{{k}}'\big( \vert \nabla u\vert\big)\,\vert\,\nabla\vert \nabla u\vert\,\vert^{2} \Big]\,d\mu\,.\nonumber
\end{align}
The standard identity 
\begin{equation}\label{MarFarVal}
\vert \nabla du\vert^{2}-\vert\,\nabla\vert \nabla u\vert\,\vert^{2}=\vert \nabla u\vert^{2} \vert \mathrm{h}\vert^{2}+\vert\,\nabla^{\top}\vert \nabla u\vert\,\vert^{2}=\vert \nabla u\vert^{2} \vert \ringg{\mathrm{h}}\vert^{2}+\vert\,\nabla^{\top}\vert \nabla u\vert\,\vert^{2}+(\mathrm{H}^2/2)\,,
\end{equation}
along with the fact that $\vert \nabla u\vert^{2}\,\mathrm{H}^2=\vert\,\nabla^{\perp}\vert \nabla u\vert\,\vert^{2}$, yields
$$
\vert \nabla u \vert^{-1}\,\vert\,\nabla\vert \nabla u\vert\,\vert^{2}\,\leq \,3 P \, ,
$$
outside the set of the critical points of $u$,
so that $\vert \nabla u \vert^{-1}\,\vert\,\nabla\vert \nabla u\vert\,\vert^{2}\in L^1_{loc}(M_o)$, whereas $|\na u|\, \eta'_{{k}}\big(\vert \nabla u\vert\big)$ is always bounded. Accordingly, as $\lim_{k \to + \infty}\eta_k'(\tau)=0$ for every $\tau \in (0, + \infty)$, the dominated convergence theorem implies that 
$$
\lim_{k\to +\infty}\,\int\limits_{E_{s}^{S}}\!\chi(u)\,\eta_{{k}}'\big( \vert \nabla u\vert\big)\,\vert\,\nabla\vert \nabla u\vert\,\vert^{2} \,d\mu\,=\,0\,.
$$
In conclusion, we obtain
\begin{align}
\int\limits_{-\infty}^{1}\!\chi'(\tau)\,\widetilde{F}(\tau)\,d\tau&\,=\,-\,\lim_{k\to +\infty}\int\limits_{M_{o}}\!\chi(u)\,\mathrm{div}( Y_{k})\,d\mu\\
&\,=\,-\,\int\limits_{M_{o}} \!\chi(u)\,\mathrm{div}( Y)\,d\mu\,=\,-\,\int\limits_{-\infty}^{1} \!\chi(\tau)\int\limits_{\{u=\tau\}}\! \!\vert \nabla u \vert^{-1}\, \mathrm{div}(Y)\,d\mathcal{H}^2 \,d\tau\,,
\end{align}
where in the last identity we used the coarea formula. It is now clear that $\widetilde{F}\in W^{1,1}_{loc}(-\infty,1)$. 
\end{proof}

\begin{remark}
It is worth pointing out that from the argument in the proof of Theorem \ref{thm_monotonicity} it also follows that the function, given by formula \eqref{eq1}, is of class $C^1$, with locally absolutely continuous first derivative.
\end{remark}

\begin{remark}
The statement of Theorem \eqref{thm_monotonicity} remains valid if $(M,g)$ is a complete, noncompact,  $P^2$-irreducible Riemannian manifold with scalar curvature greater than or equal to $-6$, and if there exists the minimal positive Green function $\mathcal{G}_o$ for $\Delta_g$, with a pole at some point $o \in M$, which vanishes at infinity. We recall that a $P^2$-irreducible manifold is a three-manifold that is irreducible (this means that every sphere bounds a ball) and contains no two-sided $ \mathbb {R} P^{2}$. Notice that this remark is significant if the manifold $M$ is nonorientable.
\end{remark}

The function $F$ is not only nondecreasing on the set $\mathcal{T}$, but we are also able to characterize the manifolds on which it is  constant almost everywhere. This is the content of the following corollary.

\begin{corollary}
\label{positive}
Under the assumptions of Theorem~\ref{thm_monotonicity}, if the function $F$ is constant on the set $\mathcal{T}$ given by equation \eqref{defmathcalT}, then $(M,g)$ is isometric to $(\HH^3, g_{\HH^3})$.
\end{corollary}
\begin{proof}
By Proposition \ref{prop:Green_existence}, we know that the function $|\na u|$ is positive in a sufficiently small punctured neighborhood of the pole $o$, thus, there exists a maximal value $L$ such that $\na u \neq 0$ in $u^{-1}(-\infty,L)$. Let $T=\mathrm{arcoth}(2-L)$. We notice that $\{u<2-\coth t\}$ is connected and $(0,T)\subseteq \mathcal{T}$. Since the function $F$ is constant in $\mathcal{T}$, one easily gets that $F'\equiv 0$ in $(0,T)$, so that all the positive summands in formula~\eqref{monoliscia} are forced to vanish for every $t \in (0,T)$. This fact has very strong implications. First of all, $\na^{\Sigma_t}|\na u| \equiv 0$ implies that $|\na u| = f (u)$, for some positive function $f:(0,T) \to (0 , + \infty)$. Such a function can be made explicit. Indeed, from~\eqref{monoliscia} one also has that $\HHH = \frac{2(2-u) f(u)}{(2-u)^2-1}$. On the other hand, from~\eqref{H} it follows that $\HHH  = -f'(u)$. All in all, we have that $f$ obeys the ODE 
$$
f'(u)\,=\,- \frac{2(2-u) }{(2-u)^2-1}\,f(u)\, .
$$
Now, the only solution to this ODE which is compatible with the asymptotic behavior of $u$ and $|\na u|$ near the pole, is given by 
$f(u) = (2-u)^2-1$. Since $u<1$ on the whole manifold, $f$ never vanishes, so that $T= +\infty$ and $|\na u| \neq 0$ everywhere.
In particular, all the level sets of $u$ are regular and diffeomorphic to each other. More precisely, by the vanishing of the Gauss-Bonnet term in~\eqref{monoliscia}, they are all diffeomorphic to a $2$-sphere and $M$ is diffeomorphic to $\R^3$. Accordingly, we have that the metric $g$ can be written on all $M \setminus \{o\}$ as 
$$
g \, = \, \frac{du  \otimes du}{\big[(2-u)^2-1\big]^2} \, +
 \, g_{\alpha \beta}
(u, \!\vartheta) \, d\vartheta^{\alpha}  \otimes d\vartheta^{\beta}\, ,
$$ 
where $g_{\alpha \beta}(u, \!\vartheta) \, d\vartheta^{\alpha}  \otimes d\vartheta^{\beta}$ represents the metric induced by $g$ on the level sets of $u$. By exploiting the vanishing of the traceless second fundamental form of the level sets in~\eqref{monoliscia}, the coefficients $g_{\alpha \beta}(u,\!\vartheta)$ obey the following first order system of PDE's
$$
\frac{\pa g_{\alpha \beta}}{\pa u}  \, = \, \frac{2(2-u) }{(2-u)^2-1}\, g_{\alpha \beta} \,  \, ,
$$
from which one can deduce
$$
g_{\alpha\beta}(u, \!\vartheta) \, d\vartheta^\alpha \otimes d\vartheta^\beta\,=\,\big[(2-u)^2-1\big]^{-1} c_{\alpha\beta}(\vartheta)\, d\vartheta^\alpha \otimes d\vartheta^\beta\,.
$$
At the same time, the traced Gauss equation together with Bochner's formula and the first identity in \eqref{MarFarVal} imply
\begin{align}
\mathrm{R}^{\{u=u_{0}\}}&=\,\mathrm{R}-2\mathrm{Ric}(\nu,\nu)-\vertl \mathrm{h}\vertr^{2}+\HHH^{2}\\
&=\,-6-2\,\vertl \nabla u\vertr^{-2}\,\mathrm{Ric}(\na u,\na u)+ ( \HHH^2/2 )\\
&=\,-6+\vertl \nabla u\vertr^{-2}\,\left[ -\Delta\vertl \nabla u\vertr^{2} +2\,\vertl \nabla du \vertr^{2}\right]+( \HHH^2/2 ) \\
&=\,-6-\vertl \nabla u\vertr^{-2}\,\Delta\vertl \nabla u\vertr^{2} +(7\,\HHH^{2}/2)\\
&=\,2\,[(2-u)^2-1]\,,
\end{align}
where we took into account that all the nonnegative summands in formula~\eqref{monoliscia} vanish on each level set of $u$ and we used the identities $\vertl \nabla u\vertr=(2-u)^{2}-1$ and $\HHH=2\,(2-u)$.
Thus, $\{u=u_{0}\}$ with the induced metric has constant sectional curvature $[(2-u)^2-1]$ and, by the vanishing of the Gauss--Bonnet term in formula~\eqref{monoliscia}, it is diffeomorphic to a $2$--sphere. 
Consequently, $(\{u=u_{0}\},g_{\{u=u_{0}\}})$ is isometric to $(\SSS^{2},[(2-u)^2-1]^{-1}g_{\mathbb{S}^2})$ by~\cite[Section~3.F]{gahula}, and, up to an isometry, one has $M\setminus \{o\}=(-\infty,1)\times \SSS^{2}$ and 
$$
g=\frac{du \otimes du}{[(2-u)^2-1]^2} \, + \frac{g_{\mathbb{S}^2}}{[(2-u)^2-1]}\,.
$$
Then, $(M\setminus \{o\},g)$ is isometric to $((0,+\infty)\times \SSS^{2},\,dr\otimes dr+\sinh^{2}r g_{\SSS^{2}})$, since
the map 
$$
(u,\vartheta)\in\left((-\infty,1)\times \SSS^{2},g\right) \mapsto \left(\mathrm{arcoth}(2-u),\vartheta\right)\in\left( (0,+\infty)\times \SSS^{2},dr\otimes dr+\sinh^{2}r g_{\SSS^{2}}\right)$$
an isometry. 
The rest of the claim then follows observing that the manifold $(M,g)$ is complete, simply connected and with constant sectional curvature $-1$ (see~\cite[Section~3.F]{gahula}, for instance).
\end{proof}


\section{A positive mass theorem in $3$D}\label{sectpositive mass theorem}

In light of the monotonicity result obtained in Theorem \ref{thm_monotonicity}, we present in this section a new positive mass theorem, Theorem \ref{thm:PMT}, for the volume-renormalized mass $m_{VR}$ on three-dimensional asymptotically hyperboloidal  manifolds. However, this result is not a direct consequence of Theorem \ref{thm_monotonicity}, due to the fact that, in order to bound from above the limit of our monotone function by a positive multiple of the mass $m_{VR}$, we require that the Green function $\mathcal{G}_o$ admits a suitable asymptotic expansion with respect to a given asymptotically hyperboloidal map $\varphi$ of order $\delta>1$. Therefore, we will first establish the positive mass inequality in a class of Riemannian manifolds strictly smaller than that of Theorem \ref{thm:PMT} and then extend the inequality to the general case by means of a density argument.

\begin{theorem}\label{thm:PMT_0version}
Let $(M,g)$ be an orientable three-dimensional asymptotically hyperboloidal manifold such that the scalar curvature satisfies $\Ro\geq -6$ and the second integral homology $H_2(M;\Z)$ does not contain any spherical classes. 
We assume that there exists the minimal positive Green function $\mathcal{G}_o$ for $\Delta_g$ with a pole at some point $o \in M$, and a distinguished asymptotically hyperboloidal map $\varphi$ of order $\delta>1$ such that 
\begin{equation}\label{estgreenfunct}
\mathcal{G}_o\,=\,\phi \,\mathcal{G}^{b}+O_2(e^{-3r}) \,,
\end{equation}
where $\varphi_* \phi(r,\xi)$ is simply a smooth positive function of $\xi\in \SSS^2$ and $\mathcal{G}^{b}$ denotes the function $(4\pi)^{-1}\big(\coth r-1\big)$.
Then,
\begin{align*}
m_{VR}(g)\geq0\,.
\end{align*}
\end{theorem}

\begin{proof}
We divide the proof in several steps.

\smallskip

\textbf{Step 1:} {\em The function $F(t)$, defined by expression \eqref{eq0}, converges to zero, as $t\to 0^+$.}

To see this fact, we recall that $u$ is related to the minimal positive  Green's function $\mathcal{G}_o$ of $(M,g)$ with pole at $o$ through the formula $u=1-4\pi\mathcal{G}_{o}$. 
Consequently, there holds
\begin{equation}\label{eq17}
\int\limits_{\Sigma_t}\vert \nabla u\vert \, d\mathcal{H}^2 \, =  \, 4\pi\,
\end{equation}
for every $t\in \mathcal{T}$.
On the other hand,  by the asymptotic behavior of $\mathcal{G}_o$ near the pole (see, for instance, \cite[Appendix]{MarRigSet}), in a sufficiently small punctured neighborhood of $o$, 
the function $u$ satisfies the bounds 
\begin{equation}\label{imp.bounds}
\frac{C_1}r\leq 1-u\leq\frac{C_2}r,
\qquad
\frac{C_3}{r^2}\leq\vert \nabla u \vert \leq \frac{C_4}{r^2},
\qquad
\vert \nabla d u \vert\leq \frac{C_5}{r^3},
\end{equation}
for some positive constants $C_{i}>0$, $i=1, \ldots, 5$, where $r$ denotes the distance to the pole $o$. Combining these bounds, we observe 
\begin{align}
\frac{\vert \nabla u \vert}{(2-u)^2-1} &\leq \frac{C_4}{r^2} \frac{1}{(1+\frac{C_1}r)^2-1}=\frac{C_4}{C_1^2(1+\frac{2r}{C_1})}\leq \frac{C_4}{C_1^2}\,,\\
\frac{\vert \nabla d u \vert}{ [(2-u)^2-1]^{1/2}\,\vert \nabla u\vert }&\leq \frac{C_5}{r^3}\,\frac{r^2}{C_3}\,\frac{1}{[(1+\frac{C_1}r)^2-1]^{1/2}}\leq \frac{C_5 }{C_1C_3}\,,
\end{align}
therefore, we conclude that 
\begin{align}
&\sinh^{2}t \int\limits_{\Sigma_t} \vert \nabla u \vert^{2} \, d\mathcal{H}^2  \,\, = \, \int\limits_{\Sigma_t} \frac{\vert \nabla u \vert}{(2-u)^2-1}\,  \vert \nabla u \vert \,  d\mathcal{H}^2   \,\, \leq \, \frac{4 \pi C_4}{C_1^2}\,,\label{feq1}\\
&\sinh t  \int\limits_{\Sigma_t}\vert \,\mathrm{H}\,\vert\,\vert \nabla u \vert\, \, d\mathcal{H}^2 \, \leq \, \int\limits_{\Sigma_t} \frac{\vert \nabla d u \vert}{ [(2-u)^2-1]^{1/2}\,\vert \nabla u\vert }\,\vert \nabla u\vert \, d\mathcal{H}^2   \,\, \leq \, \frac{4 \pi C_5}{C_1 C_3} \, , \label{feq2}
\end{align}
by equality \eqref{eq17} and inequality \eqref{eq:hdubound}.
Notice indeed that, being $u$ harmonic, every level set $\{u=2-\coth t\}$ is contained in the punctured neighborhood of $o$ in which the bounds \eqref{imp.bounds} hold true, for any $t\in (0,t_0)$ and some $t_0>0$ small enough. \\
Furthermore, estimates \eqref{imp.bounds} also imply that, for such $t\in (0,t_0)$, the sub-level sets $\{u<2-\coth t\}$ are contained in the punctured open ball $B^*_{\frac{C_2}{\coth t -1}}(o)$. Therefore, denoting by $\Omega_t$ the sub-level set $\{u<2-\coth t\}$, we get both
\begin{equation}
\int\limits_{\Omega_t}\frac{\vert \nabla u \vert}{(2-u)^2-1} \,\, d\mu\leq  \frac{C_4}{C_1^2}\,\mu(\{u<2-\coth t\})\leq \frac{C_4}{C_1^2}\,\mu(B_{\frac{C_2}{\coth t -1}}(o))\leq 
\frac{2\pi C_4}{C_1^2}\,\Big(\frac{C_2}{\coth t -1}\Big)^3
\end{equation}
and 
$$\int\limits_{\Omega_t}\frac{\vert \nabla u \vert^3}{[(2-u)^2-1]^3} \,\, d\mu\leq \frac{2\pi C_4^3}{C_1^6}\,\Big(\frac{C_2}{\coth t -1}\Big)^3$$
for any $t>0$ sufficiently small, since $\mu(B_r(o))\leq 2\pi r^3$ for every $r\in (0,r_0)$ and for some $r_0>0$ small enough, as a consequence of the limit
$$\lim_{\substack{r\to 0^+ \\ 0<r<\mathrm{inj}(o)} } \frac{\mu(B_r(o))}{\frac{4}{3}\pi r^3}=1\,.$$
Plugging these estimates into the definition of $F$, it follows then that $F(t)\to 0$, as $t \to 0^+$.

\smallskip

\textbf{Step 2:} {\em The function $F$ is nondecreasing and non-negative on the set $\mathcal{T}$, given in~\eqref{defmathcalT}.}

The claim follows directly from Theorem \ref{thm_monotonicity} together with Step 1.

\smallskip

\textbf{Step 3:} {\em Let $Q(t)$ be the function defined as
\begin{align}\label{defQ}
Q(t)&=\,4\pi t \,+\, \sinh^3t \cosh t \int\limits_{\Sigma_t}\vert \nabla u \vert^{2} \,d\mathcal{H}^2  \, - \,\sinh^2t \int\limits_{\Sigma_t}\vert \nabla u \vert\, \mathrm{H} \, d\mathcal{H}^2\,+\,2\mathrm{Vol}(\Omega_t)\,.
\end{align}
for every $t\in (0,+\infty)$. Then, $F(t)\leq Q(t)$ for all $t\in (0,+\infty)$.}

For every $t\in(0,+\infty)$, we observe that
\begin{align}
3\int\limits_{\Omega_t}&\frac{\vert \nabla u \vert}{(2-u)^2-1} \,\, d\mu\,-\int\limits_{\Omega_t}\frac{\vert \nabla u \vert^3}{[(2-u)^2-1]^3} \,\, d\mu\\
&= \,\,2\mathrm{Vol}(\Omega_t) \,\,-\,\,\int\limits_{\Omega_t}\frac{\vert \nabla u \vert}{[(2-u)^2-1]} \,\bigg(\frac{\vert \nabla u \vert}{[(2-u)^2-1]}-1\bigg)^2\, d\mu\\
&\quad\,\,-\,\,2 \int\limits_{\Omega_t}\bigg(\frac{\vert \nabla u \vert}{[(2-u)^2-1]}-1\bigg)^2\, d\mu\,.
\end{align}
Then, by comparing the expressions of the functions $F(t)$ and $Q(t)$, Step $3$ follows.

\smallskip

\textbf{Step 4:} {\em The function $Q(t)$, introduced in Step 3, satisfies
\begin{equation}\label{feq22}
\lim_{t\to+\infty} \,\frac{Q(t)}{\sinh^3 t}=\,0\,.
\end{equation}
}

Setting $\sigma= g-b$ and writing $b=dr\otimes dr+\sinh^2 r\, g_{\SSS^2}$, we obtain by formula \eqref{decaymetricatinfinity} (through direct computations) that
\begin{align}
\sigma_{rr}&=\,O(e^{-\delta r})\quad\quad &&\,\,\,\sigma_{r\alpha}=\, \sigma_{\alpha r}=\,O(e^{(1-\delta) r})\quad\quad &&\,\,\,\sigma_{\alpha\beta }=\,O(e^{(2-\delta) r})\\
{}^g\Gamma_{rr}^r&=\,O(e^{-\delta r})\quad\quad &&{}^g\Gamma_{r\alpha}^r=\,O(e^{(1-\delta)r})\quad\quad &&{}^g\Gamma_{\alpha\beta}^r=\, {}^b\Gamma_{\alpha\beta}^r+O(e^{(2-\delta)r})\\
{}^g\Gamma_{rr}^\lambda&=\,O(e^{-(1+\delta) r})\quad\quad\quad&& {}^g\Gamma_{r\alpha}^\beta=\,{}^b\Gamma_{r\alpha}^\beta+O(e^{-\delta r})\quad\quad \quad && {}^g\Gamma_{\alpha\beta}^\lambda=\,{}^b\Gamma_{\alpha\beta}^\lambda +O(e^{(1-\delta) r})
\end{align}
(by convention, Greek indices refer to spherical coordinates and vary in the set $\{1,2\}$).
Recalling the link $u=1-4\pi\mathcal{G}_{o}$ between $u$ and $\mathcal{G}_{o}$, we find among the direct consequences of \eqref{estgreenfunct} the existence of some $t_0\in (0,+\infty)$ such that $(t_0,+\infty)\subset \mathcal{T}$, and the identities
\begin{align}
&\frac{|\nabla u|}{(2-u)^2-1}=\,1+O(e^{-r}) \label{usefulestimate1}\\
&\quad\quad\,\,\mathrm{H}=\,2\,\frac{\cosh r}{\sinh r}\,\big(1+O(e^{-r})\big)\,. \label{feqestH}
\end{align}
Notice also that $\Sigma_t$ is connected for every $t\in (t_0,+\infty)$ (by following a similar argument to that in \cite[Remark 2.1]{AMO}).
Putting together equalities \eqref{eq17} and \eqref{usefulestimate1}, we conclude that, unless we pass a bigger $t_0>0$, 
\begin{equation}
0\leq\, \cosh t \int\limits_{\Sigma_t}\vert \nabla u \vert^{2} \,d\mathcal{H}^2=\,\frac{\cosh t}{\sinh^2 t} \int\limits_{\Sigma_t} \frac{\vert \nabla u \vert^{2}}{(2-u)^2-1} \,d\mathcal{H}^2\leq\, 8\pi \,\frac{\cosh t}{\sinh^2 t}\,,
\end{equation}
for all $t\geq t_0$. At the same time, the combination of equalities \eqref{eq17} and \eqref{feqestH} implies that, unless we choose a bigger $t_0>0$, 
\begin{equation}
\bigg\vert\,\int\limits_{\Sigma_t}\vert \nabla u \vert\, \mathrm{H} \, d\mathcal{H}^2\,\bigg\vert \leq \int\limits_{\Sigma_t}\vert \nabla u \vert\, |\mathrm{H}| \, d\mathcal{H}^2\leq\, 16\pi\,,
\end{equation}
for all $t\geq t_0$. Thus,
\begin{equation}
\lim_{t\to+\infty} \,\sinh^{\!-3} t\,\Bigg[4\pi t \,+\, \sinh^3t \cosh t \int\limits_{\Sigma_t}\vert \nabla u \vert^{2} \,d\mathcal{H}^2  \, - \,\sinh^2t \int\limits_{\Sigma_t}\vert \nabla u \vert\, \mathrm{H} \, d\mathcal{H}^2\Bigg]=\,0\,.
\end{equation}
Lastly, by applying l'Hospital's rule and using estimate \eqref{usefulestimate1} once more, we obtain
\begin{align*}
\lim_{t\to +\infty}  \frac{\mathrm{Vol}(\Omega_t)}{\sinh^3 t}&=\,\lim_{t\to +\infty}\,\frac{ \sinh^{\!-2} t \int\limits_{\Sigma_t}|\nabla u|^{-1} \,d\mathcal{H}^2}{2\sinh^2 t \cosh t}\\
&=\,\lim_{t\to +\infty}\,\frac{1}{\cosh t}\int\limits_{\Sigma_t}\bigg(\frac{(2-u)^2-1}{|\nabla u|}\bigg)^{\!2}\,|\nabla u| \,d\mathcal{H}^2 =\,0\,.
\end{align*}

\smallskip

\textbf{Step 5:} {\em There holds that
\begin{equation}\label{firstupperboundlimit}
\lim_{t\to +\infty}F(t)\leq \limsup_{t\to +\infty} \Bigg(4\pi t \,+\,2\mathrm{Vol}(\Omega_t) \,-\,\frac{1}{4}\,\frac{\sinh t}{\cosh t}\int\limits_{\Sigma_t}\mathrm{H}^2 \,d\mathcal{H}^2\Bigg)\,.
\end{equation}
Hence, defining
\begin{align}
Q_1(t)&=4\pi t \,+\,2\mathrm{Vol}_{hyp}(D_t) \,-\,\frac{1}{4}\,\frac{\sinh t}{\cosh t} \int\limits_{\Sigma_t}\mathrm{H}^2_b \,d\mathcal{H}^2_b\,,\label{defQ1}\\
Q_2(t)&=\,2\Big(\mathrm{Vol}(\Omega_t)-\mathrm{Vol}_{hyp}(D_t)\Big)\,-\,\frac{1}{4}\,\frac{\sinh t}{\cosh t}\Bigg(\, \int\limits_{\Sigma_t}\mathrm{H}^2 \,d\mathcal{H}^2-\int\limits_{\Sigma_t}\mathrm{H}^2_b \,d\mathcal{H}^2_b\Bigg)\,,\label{defQ2}
\end{align}
for any $t\in (t_0,+\infty)$ and some $t_0\in (0,+\infty)$ sufficiently big, we have
\begin{equation}
\lim_{t\to +\infty}F(t)\leq \,\limsup_{t\to +\infty} \Big(Q_1(t)\,+\,Q_2(t)\Big)\,.\label{ffeq60}
\end{equation}
Here, $D_t$ is the compact domain in $\HH^3$ having $S_t=\{\varphi_*u=2-\coth t\}$ as a boundary.
}

We start by recalling that in the Step 4 we saw the existence of some $t_0\in (0,+\infty)$ such that $(t_0,+\infty)\subset \mathcal{T}$.
Next, let us rewrite the function $Q(t)$ in the following way
\begin{align}
Q(t)&=\,4\pi t \,+\,2\mathrm{Vol}(\Omega_t) \,-\,\frac{1}{4}\,\frac{\sinh t}{\cosh t}\int\limits_{\Sigma_t} \mathrm{H}^2 \,d\mathcal{H}^2\\
&\quad +\,\frac{1}{4}\,\frac{\sinh t}{\cosh t}\int\limits_{\Sigma_t}\bigg(\,\mathrm{H}-\frac{2(2-u)}{(2-u)^2-1} \,\vert \nabla u \vert\bigg)^{\!2}d\mathcal{H}^2\,. \label{eq21}
\end{align}
Lastly, we notice that
$$\lim_{t\to+\infty} \frac{F(t)}{\sinh^3 t}=0\,,$$
as a consequence of Step 3 and Step 4.
Thus, we can apply l'Hospital's rule \cite[Theorem II]{taylor1} and have
\begin{align}
\lim_{t\to +\infty}F(t)&=\,\lim_{t\to +\infty} \frac{ \sinh^{\!-3} t \, F(t)}{ \sinh^{\!-3} t}\\
&\leq\,\limsup_{t\to +\infty} \Big(-\,\frac{\sinh t}{3\cosh t}\, F'(t)\,+\, F(t)\Big)\\
&\leq \,\limsup_{t\to +\infty} \Bigg[-\,\frac{1}{4}\,\frac{\sinh t}{\cosh t}\, \int\limits_{\Sigma_t}\!\bigg(\,\mathrm{H}-\frac{2(2-u)}{(2-u)^2-1} \,\vert \nabla u \vert\bigg)^{\!2}d\mathcal{H}^2\,+\, Q(t)\Bigg]\\
&\leq \,\limsup_{t\to +\infty} \Bigg(4\pi t \,+\,2\mathrm{Vol}(\Omega_t) \,-\,\frac{1}{4}\,\frac{\sinh t}{\cosh t}\int\limits_{\Sigma_t}\mathrm{H}^2 \,d\mathcal{H}^2\Bigg)\,,
\end{align}
 where the second inequality is a consequence of the expression of $F'(t)$, given in formula \eqref{monoliscia}, as sum of nonnegative terms, and the third inequality of equality \eqref{eq21}.

\smallskip

\textbf{Step 6:} {\em The function $Q_1(t)$, defined by expression \eqref{defQ1} on the interval $(t_0,+\infty)$, satisfies
\begin{equation}\label{upperboundoflimsupofQ1}
 \limsup_{t\to +\infty} \,Q_1(t)\leq\, 0\,.
\end{equation}
}

To obtain the upper limit \eqref{upperboundoflimsupofQ1}, we recall two classical inequalities that hold in the hyperbolic space $(\Hyp^3,g_{hyp})$. The first one is a Willmore-type inequality, which states that
\begin{equation}\label{Wiltypeineq}
\frac{1}{4}\int\limits_{\Sigma}\mathrm{H}^2_{g_{hyp}} \,d\mathcal{H}^2_{g_{hyp}}\geq \,4\pi+\mathrm{Area}_{g_{hyp}}(\Sigma)\,,
\end{equation}
for any closed surface $\Sigma$ in the hyperbolic space. It follows due to the conformal invariance of the Willmore functional observed in \cite{Chen} (for some generalizations, see also \cite{Chai, Rit,Sch}).
The second one is a isoperimetric inequality, which asserts that a geodesic sphere of $(\Hyp^3,g_{hyp})$ has the smallest area among all closed surfaces in $(\Hyp^3,g_{hyp})$ that enclose the same amount of volume, see \cite{Schmidt}.
Furthermore, it is immediate to see the equality
$$
\int\limits_{\Sigma_t} \mathrm{H}^2_b \,d\mathcal{H}^2_b\,=\,\int\limits_{S_t}\mathrm{H}^2_{hyp} \,d\mathcal{H}^2_{hyp}\,,
$$
which, together with inequality \eqref{Wiltypeineq}, implies
\begin{equation}
Q_1(t)\leq\, 4\pi t \,+\,2\mathrm{Vol}_{hyp}(D_t) \,-\,\,\frac{\sinh t}{\cosh t} \Big(4\pi +\mathrm{Area}_{hyp}(S_t)\Big) \,.
\end{equation}
At the same time, applying the isoperimetric inequality, we obtain
$$\mathrm{Area}_{hyp}(S_t)\,\geq\, 4\pi \sinh^2 \! R_t\,, $$
where $R_t$ is the radius of a geodesic ball of $(\Hyp^3,g_{hyp})$ whose volume is equal to $\mathrm{Vol}_{hyp}(D_t)$. More precisely, $R_t$ is defined by the identity
\begin{equation}
\mathrm{Vol}_{hyp}(D_t)\,=\,2\pi \sinh R_t \cosh R_t\,-\,2\pi R_t\,.
\end{equation}
Thus, joining all, we have
\begin{align}
Q_1(t)&\leq\, -\,4\pi\,\frac{\sinh t}{\cosh t}\,+\, 4\pi t \,+\,2 \big( 2\pi \sinh R_t \cosh R_t\,-\,2\pi R_t\big) \,-\,\,4\pi \sinh^2  \!R_t\,\frac{\sinh t}{\cosh t}\\
&=\,2\pi\,\frac{e^{2(R_t-t)}}{1+e^{-2t}}\,+\,4\pi(t-R_t)\,-\,2\pi\,+\,O(e^{-2t})\,+\,O(e^{-2R_t})\,.
\end{align}
The statement then follows once we have shown that $R_t=t+o_t(1)$. With this aim, let us observe that
\begin{align*}
\mathrm{Vol}_{hyp}(D_t)&=\,\mathrm{Vol}_b(\{2-\coth t_0<u<2-\coth t\})\,+\,C_0 \\
&=\int_{t_0}^t d\tau\, \sinh^2 \tau \int\limits_{\Sigma_\tau} \bigg(\frac{(2-u)^2-1}{|\nabla^b u|_b}\bigg)^{\!2}\,|\nabla^b u|_b \,d\mathcal{H}^2_b\,+\,C_0\,,
\end{align*}
for any $t\in (t_0,+\infty)$ and some $t_0>0$ such that $(t_0,+\infty)\subset \mathcal{T}$. Here, $C_0$ is a positive constant independent of $t$ and the second equality is achieved by means of the coarea formula.  \\
Now, the asymptotic behaviors introduced at the beginning of Step 4 and the relation between $u$ and $\mathcal{G}_o$ given by the formula $u=1-4\pi\mathcal{G}_{o}$ lead to the estimates
\begin{align}
|\nabla^b u|_b&=\,\big(1+O(e^{-\delta r}) \big)\,|\nabla u| \\
d\mathcal{H}^2_b&=\,\big(1+O(e^{-\delta r}) \big)\,d\mathcal{H}^2\,,\label{ffeq57}
\end{align}
which, along with equality \eqref{usefulestimate1}, yield
\begin{equation}
\bigg(\frac{(2-u)^2-1}{|\nabla^b u|_b}\bigg)^{\!2}\,|\nabla^b u|_b \,d\mathcal{H}^2_b\,=\,\big(1+O(e^{-r})\big)|\nabla u|\,d\mathcal{H}^2\,.
\end{equation}
Moreover, taking advantage from the fact that the function $\phi$ in \eqref{estgreenfunct} is smooth and positive on $\SSS^2$, we get the existence of two positive constants $c_1$ and $c_2$ such that $c_1t\leq r(p)\leq c_2 t$ for all $p\in \Sigma_t$, for every $t>0$ sufficiently big. Thus, for $t>0$ large enough,
\begin{equation}\label{estimatesonradialcoordinates}
c_1 t\,\leq\, \min\limits_{\Sigma_t}\,r\,\leq\, \max\limits_{\Sigma_t}\,r\,\leq\, c_2 t \,.
\end{equation}
This implies that $B^{hyp}_{c_1 t}\subset D_t\subset B^{hyp}_{c_2 t}$, which in turn gives
\begin{equation}\label{estimatesonRt}
c_1 t\,\leq \,R_t\,\leq\, c_2 t\,.
\end{equation}
Then, joining all these facts with identity \eqref{eq17}, we obtain  
$$ \int\limits_{\Sigma_t}\bigg(\frac{(2-u)^2-1}{|\nabla^b u|_b}\bigg)^{\!2}\,|\nabla^b u|_b \,d\mathcal{H}^2_b\,=\,4\pi\big(1+O(e^{-\varepsilon t}) \big)$$
for some $\varepsilon>0$. Consequently, we have 
\begin{equation}
\mathrm{Vol}_{hyp}(D_t)\,=\,2\pi \sinh t \cosh t\,-\,2\pi t +O(e^{\max\{2-\varepsilon, 0\} t})\,.
\end{equation}
Then, it follows that
\begin{align}
2\pi \sinh R_t \cosh R_t\,-\,2\pi R_t&=\,\big(2\pi \sinh t \cosh t\,-\,2\pi t\big) (1+o_t(1))\\
\frac{\sinh R_t \cosh R_t\,-\,R_t}{ \sinh t\cosh t\,-\, t}&=1+o_t(1)\\
\frac{e^{2R_t}(1-4R_te^{-2R_t}-e^{-4R_t})}{ e^{2t}(1-4te^{-2t}-e^{-4t})}&=1+o_t(1)\,.
\end{align}
By estimates \eqref{estimatesonRt}, we thus conclude that
\begin{align}
\frac{e^{2R_t}(1+o_t(1))}{ e^{2t}(1+o_t(1))}&=1+o_t(1)\\
e^{2(R_t-t)}&=1+o_t(1)\\
R_t-t&=o_t(1)\,.
\end{align}

\smallskip

\textbf{Step 7:} {\em The function $Q_2(t)$, defined by expression \eqref{defQ2} on the interval $(t_0,+\infty)$, satisfies
\begin{equation}\label{upperboundoflimsupofQ2}
 \limsup_{t\to +\infty} Q_2(t)\,=\,\frac{1}{2}\,m_{VR}(g)\,.
\end{equation}
}

In the same spirit as in \cite{AMMO} and following similar computations, we are going to compare the expression of the Willmore functional along the level sets of the function $u$ with
an analogous expression in which the geometric quantities are computed with respect to the background metric $b$. To this end, we observe that
\begin{align}
 \det g&=\,\det b\,\big(1+ \mathrm{tr}_b(\sigma)+O(e^{-2\delta r})\big)\\
\sqrt{\det g}&=\,\sqrt{\det b}\,\big(1+ \frac{1}{2}\mathrm{tr}_b(\sigma)+O(e^{-2\delta r})\big)\\
g^{rr}&=\,b^{rr}-\sigma_{rr}+O(e^{-2\delta r})\\
g^{r\alpha}&=\,b^{r\alpha}-b^{\alpha \beta}\sigma_{\beta r}b^{rr}+O(e^{-(1+2\delta) r})\\
g^{\alpha\beta}&=\,b^{\alpha\beta}-b^{\alpha \gamma}\sigma_{\gamma \lambda}b^{\lambda \beta}+O(e^{-2(1+\delta) r})\,.
 \end{align}
Let $\nu=\nabla u /|\nabla u|$ and $\nu_b=\nabla^b u /|\nabla^b u|_b$. We notice that
\begin{equation}
\nu_b^r=\,1+O(e^{-r})\quad\quad\text{and}\quad\quad \nu_b^\alpha=\,O(e^{-2r})\,.
\end{equation}
Therefore, by direct computations, we get
\begin{align}
\nu&=\,\Big(\nu_b^r+\frac{1}{2} \nu_b^r \sigma(\nu_b,\nu_b)-b^{rr}\sigma_{rr} \nu_b^r-b^{rr}\sigma_{r\beta} \nu_b^\beta+O(e^{-2\delta r})\Big) \partial_r\\
&\quad +\Big(\nu_b^\alpha+\frac{1}{2} \nu_b^\alpha \sigma(\nu_b,\nu_b)-b^{\alpha\beta}\sigma_{\beta r} \nu_b^r-b^{\alpha \gamma}\sigma_{\gamma \lambda} \nu_b^\lambda+O(e^{-(1+2\delta )r})\Big) \partial_\alpha\,,
\end{align}
which implies 
\begin{equation}\label{estimateareaelement}
d\mathcal{H}^2=\,\Big(1+\frac{1}{2} \mathrm{tr}_b(\sigma)-\frac{1}{2}\sigma(\nu_b,\nu_b)+O(e^{-2\delta r})\Big)\,d\mathcal{H}^2_b\,.
\end{equation}
Since $\mathrm{H}=(g^{ij}-\nu^i\nu^j)(\nabla du)_{ij}/|\nabla u|$, we estimate
\begin{align}
g^{rr}-\nu^r\nu^r&=\,\epsilon^{rr}-\epsilon^{ri}\sigma_{ij}\epsilon^{jr}+O(e^{-2\delta r})\\
g^{r\alpha}-\nu^r\nu^\alpha&=\,\epsilon^{r\alpha}-\epsilon^{ri}\sigma_{ij}\epsilon^{j\alpha}+O(e^{-(1+2\delta )r})\\
g^{\alpha\beta}-\nu^\alpha\nu^\beta&=\,\epsilon^{\alpha\beta}-\epsilon^{\alpha i}\sigma_{ij}\epsilon^{j\beta}+O(e^{-2(1+\delta )r})\,,
\end{align}
where $\epsilon^{ij}=b^{ij}-\nu_b^i\nu_b^j$, for all $i,j\in\{r,1,2\}$, and 
\begin{equation}
\epsilon^{rr}=\,O(e^{- r})\quad\quad\quad \epsilon^{r\alpha}=\,O(e^{-2r})\quad\quad\quad \epsilon^{\alpha\beta}=\,b^{\alpha\beta}+O(e^{-4r})\,.
\end{equation}
For simplicity, we continue to denote by $\mathrm{h}^b$ its extension $\nabla^b du/|\nabla^b u|_b$.
Setting $\omega=du /|\nabla^b u|_b$ and $\mathrm{D}_{g,b}\Gamma_{ij}^k={}^g\Gamma_{ij}^k -{}^b\Gamma_{ij}^k $, the Willmore energy integrand then satisfies
\begin{align}
\mathrm{H}^2\,d\mathcal{H}^2&=\,\Big( \mathrm{H}_b^2\,+ \sigma(\nu_b,\nu_b)\mathrm{H}_b^2+\frac{1}{2}\,\epsilon^{ij} \sigma_{ij}\mathrm{H}_b^2 -2\mathrm{H}_b\epsilon^{ij}\mathrm{D}_{g,b}\Gamma_{ij}^k \,\omega_k-2 \mathrm{H}_b\epsilon^{ik}\sigma_{kl}\epsilon^{lj}\mathrm{h}^b_{ij}+O(e^{-2\delta r})\Big)\,d\mathcal{H}^2_b\nonumber\\
&=\,\Big( \mathrm{H}_b^2\,+2 \sigma(\nu_b,\nu_b)\mathrm{H}_b+2\,\epsilon^{ij} \sigma_{ij} -4\epsilon^{ij}\mathrm{D}_{g,b}\Gamma_{ij}^k \,g_{kl}\nu_b^l-4\epsilon^{ik}\sigma_{kl}\epsilon^{lj}\mathrm{h}^b_{ij} \nonumber\\
&\,\,\quad + \sigma(\nu_b,\nu_b)\mathrm{H}_b \big(\mathrm{H}_b-2)+\frac{1}{2}\,\epsilon^{ij} \sigma_{ij}\big(\mathrm{H}_b^2-4\big) -2\big(\mathrm{H}_b-2)\epsilon^{ij}\mathrm{D}_{g,b}\Gamma_{ij}^k \,\omega_k
+4\epsilon^{ij}\mathrm{D}_{g,b}\Gamma_{ij}^k\,\sigma_{kl}\nu_{b}^l \nonumber\\
&\,\,\quad-2 \big(\mathrm{H}_b-2\big)\epsilon^{ik}\sigma_{kl}\epsilon^{lj}\mathrm{h}^b_{ij}+O(e^{-2\delta r})\Big)\,d\mathcal{H}^2_b\nonumber\\
&=\,\Big( \mathrm{H}_b^2\,+2 \sigma(\nu_b,\nu_b)\mathrm{H}_b+2\,\epsilon^{ij} \sigma_{ij} -4\epsilon^{ij}\mathrm{D}_{g,b}\Gamma_{ij}^k \,g_{kl}\nu_b^l-4\epsilon^{ik}\sigma_{kl}\epsilon^{lj}\mathrm{h}^b_{ij}
+O(e^{-(1+\delta) r})\Big)\,d\mathcal{H}^2_b\,, \label{ffeq55}
\end{align}
by virtue of the fact that $\mathrm{H}_b=2+O(e^{-r})$.
Now, we observe that
\begin{equation}
\mathrm{div}^{\top}_b(X^\top)+\sigma(\nu_b,\nu_b)\mathrm{H}_b\,=\,\mathrm{div}^{\top}_b(X)\,=\,\epsilon^{ij}\Big[\partial_i(\sigma_{jk} \nu_b^k)-{}^b\Gamma_{ij}^k\sigma_{kl} \nu_b^l\Big]\,,
\end{equation}
where $X=b^{ij}\sigma_{jk}\nu_b^k \partial_i$, which is the vector field obtained from $\sigma(\nu_b, \cdot)$ by raising an index with respect the metric $b$, and $X^\top=X-b(X, \nu_b)\nu_b$. Thus, there holds
\begin{align}
\sigma(\nu_b,\nu_b)\mathrm{H}_b&=\,\epsilon^{ij}\Big[\partial_i(\sigma_{jk} \nu_b^k)-{}^b\Gamma_{ij}^k\sigma_{kl} \nu_b^l\Big]-\mathrm{div}^{\top}_b(X^\top)\\
&=\,\epsilon^{ij} \Big[\Big(\mathrm{D}_{g,b}\Gamma_{ik}^l g_{jl}+\mathrm{D}_{g,b}\Gamma_{ij}^l g_{kl}\Big)\nu_b^k+\epsilon^{kl}\sigma_{jk}\mathrm{h}^b_{il}\Big]-\mathrm{div}^{\top}_b(X^\top)\,,
\end{align}
which implies 
\begin{align}
& \sigma(\nu_b,\nu_b)\mathrm{H}_b+\epsilon^{ij} \sigma_{ij} -2\epsilon^{ij}\mathrm{D}_{g,b}\Gamma_{ij}^k \,g_{kl}\nu_b^l-2\epsilon^{ik}\sigma_{kl}\epsilon^{lj}\mathrm{h}^b_{ij}\\
&\quad=\,\epsilon^{ij} \mathrm{D}_{g,b}\Gamma_{ik}^l g_{lj}\nu_b^k-\epsilon^{ij}\mathrm{D}_{g,b}\Gamma_{ij}^l \,g_{lk}\nu_b^k+\epsilon^{ij} \sigma_{ij}-\epsilon^{ik}\sigma_{kl}\epsilon^{lj}\mathrm{h}^b_{ij}-\mathrm{div}^{\top}_b(X^\top)\\
&\quad=\,b^{ij}\Big( \mathrm{D}_{g,b}\Gamma_{ik}^l g_{lj}-\mathrm{D}_{g,b}\Gamma_{ij}^l \,g_{lk}\Big)\nu_b^k+\epsilon^{ij} \sigma_{ij}-\epsilon^{ik}\sigma_{kl}\epsilon^{lj}\mathrm{h}^b_{ij}-\mathrm{div}^{\top}_b(X^\top)\\
&\quad=\,b^{ij}\Big( \mathrm{D}_{g,b}\Gamma_{ik}^l g_{lj}-\mathrm{D}_{g,b}\Gamma_{ij}^l \,g_{lk}\Big)\nu_b^k-\mathrm{div}^{\top}_b(X^\top)+O(e^{-(1+\delta) r})\,,\label{ffeq56}
 \end{align}
since we have
\begin{align}
\epsilon^{ij} \sigma_{ij}&=\,\epsilon^{\alpha\beta}\sigma_{\alpha\beta}+O(e^{-(1+\delta) r})\\
\epsilon^{ik}\sigma_{kl}\epsilon^{lj}\mathrm{h}^b_{ij}&=\,\epsilon^{\lambda\beta}\sigma_{\lambda\beta}+O(e^{-(1+\delta) r})\,.
\end{align}
Then, plugging information \eqref{ffeq56} in formula \eqref{ffeq55}, the expression for the Willmore energy integrand becomes
\begin{align} 
\mathrm{H}^2\,d\mathcal{H}^2&=\,\Big[ \mathrm{H}_b^2\,+2 b^{ij}\Big( \mathrm{D}_{g,b}\Gamma_{ik}^l g_{lj}-\mathrm{D}_{g,b}\Gamma_{ij}^l \,g_{lk}\Big)\nu_b^k-2\mathrm{div}^{\top}_b(X^\top)+O(e^{-(1+\delta) r})\Big]\,d\mathcal{H}^2_b\,,\quad\quad \label{estfffeq1}
\end{align}
where there holds
\begin{equation}\label{estfffeq2}
 b^{ij}\Big( \mathrm{D}_{g,b}\Gamma_{ik}^l g_{lj}-\mathrm{D}_{g,b}\Gamma_{ij}^l \,g_{lk}\Big)\nu_b^k\,=\,O(e^{-\delta r})\,.
 \end{equation}
Accordingly, we conclude by the divergence theorem that
\begin{align*}
\int\limits_{\Sigma_t} \mathrm{H}^2 \,d\mathcal{H}^2-\int\limits_{\Sigma_t} \mathrm{H}^2_b \,d\mathcal{H}^2_b
  &=\,\int\limits_{\Sigma_t}\Big[ 2 b^{ij}\Big( \mathrm{D}_{g,b}\Gamma_{ik}^l g_{lj}-\mathrm{D}_{g,b}\Gamma_{ij}^l \,g_{lk}\Big)\nu_b^k-2\mathrm{div}^{\top}_b(X^\top)+O(e^{-(1+\delta) r})\Big]\,d\mathcal{H}^2_b\\
&=\,\int\limits_{\Sigma_t}\Big[ 2 b^{ij}\Big( \mathrm{D}_{g,b}\Gamma_{ik}^l g_{lj}-\mathrm{D}_{g,b}\Gamma_{ij}^l \,g_{lk}\Big)\nu_b^k+O(e^{-(1+\delta) r})\Big]\,d\mathcal{H}^2_b\,.
\end{align*}
Now, we claim that
\begin{equation}\label{ffeq58}
\int\limits_{\Sigma_t}O(e^{-(1+\delta) r})\,d\mathcal{H}^2_b\,=\,O(e^{-\varepsilon t})\,,
\end{equation}
for some $\varepsilon>0$.
A crucial fact is that $\delta>1$. Indeed, it is convenient to multiply and divide the integrand for $|\nabla u|$, as 
\begin{equation}
|\nabla u|\,=\,\frac{\phi}{\sinh^2 r}\big(1+O(e^{-r})\big)\,.\label{normofgradientuestimate}
\end{equation}
Hence, by equality \eqref{ffeq57}, we can write the integrad in \eqref{ffeq58} as
$$O(e^{-(1+\delta) r})\,d\mathcal{H}^2_b\,=\,O(e^{-(\delta-1) r})\,|\nabla u|\,d\mathcal{H}^2_b\,=\,O(e^{-(\delta-1) r})\,|\nabla u|\,d\mathcal{H}^2\,,$$
which, together with equality \eqref{eq17} and the estimates \eqref{estimatesonradialcoordinates}, leads to the desired claim.\\
By combining these recent results and recalling expression \eqref{defQ2} of the function $Q_2$, we have
\begin{align}
&\frac{1}{4}\,\frac{\sinh t}{\cosh t}\,\Bigg(\,\int\limits_{\Sigma_t} \mathrm{H}^2 \,d\mathcal{H}^2\,-\,\int\limits_{\Sigma_t} \mathrm{H}^2_b \,d\mathcal{H}^2_b\Bigg)\\
&\quad \,\, =\,\frac{1}{2}\,\big(1+O(e^{-2t})\big)\,\Bigg(\,\int\limits_{\Sigma_t} b^{ij}\Big( \mathrm{D}_{g,b}\Gamma_{ik}^l g_{lj}-\mathrm{D}_{g,b}\Gamma_{ij}^l \,g_{lk}\Big)\nu_b^k\,d\mathcal{H}^2_b+O(e^{-\varepsilon t})\Bigg)\,.
\end{align}
Let us show that
\begin{equation}
O(e^{-2t}) \int\limits_{\Sigma_t}b^{ij}\Big( \mathrm{D}_{g,b}\Gamma_{ik}^l g_{lj}-\mathrm{D}_{g,b}\Gamma_{ij}^l \,g_{lk}\Big)\nu_b^k\,d\mathcal{H}^2_b \,=\,O(e^{-\varepsilon t})\,,
\end{equation}
unless we pass a smaller $\varepsilon>0$. In this case, it is convenient to rewrite the expression as
\begin{align}
&O(e^{-2t})\int\limits_{\Sigma_t} b^{ij}\Big( \mathrm{D}_{g,b}\Gamma_{ik}^l g_{lj}-\mathrm{D}_{g,b}\Gamma_{ij}^l \,g_{lk}\Big)\nu_b^k\,d\mathcal{H}^2_b \\
&\quad\,\,=\,O(e^{-2t})\sinh^2 t\int\limits_{\Sigma_t}\frac{(2-u)^2-1}{|\nabla u|}\,\Big[b^{ij}\Big( \mathrm{D}_{g,b}\Gamma_{ik}^l g_{lj}-\mathrm{D}_{g,b}\Gamma_{ij}^l \,g_{lk}\Big)\nu_b^k\Big] \,|\nabla u|\,d\mathcal{H}^2_b\,,
\end{align}
therefore, the estimates \eqref{usefulestimate1} and \eqref{estfffeq2}, together with \eqref{ffeq57}, lead to
\begin{align}
O(e^{-2t})\int\limits_{\Sigma_t} b^{ij}\Big( \mathrm{D}_{g,b}\Gamma_{ik}^l g_{lj}-\mathrm{D}_{g,b}\Gamma_{ij}^l \,g_{lk}\Big)\nu_b^k\,d\mathcal{H}^2_b &=\,O(1) \int\limits_{\Sigma_t} O(e^{-\delta r}) \,|\nabla u|\,d\mathcal{H}^2_b\\
&=\,O(1) \int\limits_{\Sigma_t} O(e^{-\delta r}) \,|\nabla u|\,d\mathcal{H}^2\,.
\end{align}
The wanted statement then follows from equality \eqref{eq17} with estimate \eqref{estimatesonradialcoordinates}, and we conclude 
\begin{align}
\frac{1}{4}\,\frac{\sinh t}{\cosh t}\,\Bigg(\,\int\limits_{\Sigma_t} \mathrm{H}^2 \,d\mathcal{H}^2\,-\,\int\limits_{\Sigma_t}\mathrm{H}^2_b \,d\mathcal{H}^2_b\Bigg)& =\,\frac{1}{2} \,\int\limits_{\Sigma_t} b^{ij}\Big( \mathrm{D}_{g,b}\Gamma_{ik}^l g_{lj}-\mathrm{D}_{g,b}\Gamma_{ij}^l \,g_{lk}\Big)\nu_b^k\,d\mathcal{H}^2_b+O(e^{-\varepsilon t})\\
&=-\,\frac{1}{2} \,\int\limits_{\Sigma_t} (\mathrm{div}_b(g)-d\mathrm{tr}_b(g)) (\nu_b)\,d\mathcal{H}^2_b+O(e^{-\varepsilon t})\,.
\end{align}
This implies that
\begin{equation}
Q_2(t)\,=\,2\Big(\mathrm{Vol}(\Omega_t)-\mathrm{Vol}_{hyp}(D_t)\Big)\,+\,\frac{1}{2} \int\limits_{\Sigma_t}(\mathrm{div}_b(g)-d\mathrm{tr}_b(g)) (\nu_b)\,d\mathcal{H}^2_b+O(e^{-\varepsilon t})\,.
\end{equation}
By Remark \ref{importantpropertyrenvol}, this equality yields
\begin{equation}
 \limsup_{t\to +\infty} Q_2(t)\,=\,\frac{1}{2}\,m_{VR}(g)\,.
\end{equation}

\smallskip

\textbf{Step 8:} {\em The positive volume-renormalized mass inequality, that is $m_{VR}(g)\geq0$, holds}

We know that $(t_0,+\infty)\subset \mathcal{T}$, for some $t_0\in (0,+\infty)$. By Step 2, we then have $$\lim_{t\to +\infty}F(t)\geq 0\,.$$
Putting together this last result with inequality \eqref{ffeq60}, Step 6 and Step 7, we obtain the positive volume-renormalized mass inequality.
\end{proof}

\begin{remark}\label{diffexpbutthesametheo}
Going through the previous proof, we see that the behavior of the error term in the asymptotic expansion of the Green function $\mathcal{G}_o$, which is necessary to obtain Step 7, depends on how rapidly the metric tends to hyperbolic metric in an asymptotically hyperboloidal map.
\end{remark}

To carry out the density argument and recover the positive mass theorem in the general case, we need the following key result that holds for generic asymptotically hyperbolic manifolds of dimension greater than or equal to $3$.

\begin{theorem}\label{Yamabe_lowreg}
Let $(M^{n+1},\hat{g})$ be an asymptotically hyperbolic manifold of class $C^{2,\alpha}$ of dimension $n+1 \geq 3$. Assume that $\hat{g}$ has scalar curvature $-n(n+1)$ in a neighborhood of the infinity boundary.
Fix $\delta\in (0,2)$ and let $g$ be another Riemannian metric on $M$ such that $g-\hat{g}\in C^{2,\alpha}_{\delta}(M; S^2 T^*M)$, where $ S^2 T^*M$ is the bundle of symmetric $(0,2)$--tensors. Then there is a unique positive function $\phi-1\in C^{2,\alpha}_{\delta}(M)$ such that the metric $\overline{g}=\phi^{\frac{4}{n-1}}g$ has constant scalar curvature $-n(n+1)$.
\end{theorem}

\begin{remark}
Observe that this extends the solution of the Yamabe problem to metrics which are conformally compact of lower regularity. Indeed, for $\delta$ and $g$ as in the theorem, the metric $g$  will only be $C^{k,\alpha}$-conformally compact, where $k=\lfloor \delta \rfloor\in \left\{0,1\right\}$ and $\alpha=\delta- \lfloor \delta \rfloor$.
\end{remark}
\begin{proof}[Proof of Theorem \ref{Yamabe_lowreg}]
Fix $\hat{g}$ as in the theorem and consider the set of Riemannian metrics
\begin{align*}
\mathcal{M}\,=\,\left\{
g\in C^{2,\alpha}(M; S^2_+T^*M)\,:\, g-\hat{g}\in C^{2,\alpha}_{\delta}(M; S^2 T^*M)\right\}\,.
\end{align*}
Let furthermore $\mathcal{U}$ be the set of metrics in $\mathcal{M}$ which are $C^{2,\alpha}$-conformally compact and note that $\mathcal{U}\subset \mathcal{M}$ is a dense subset. By \cite[Proposition 4.3]{DaKrMc}, the Yamabe problem can be solved on $\mathcal{U}$. Therefore, for every $g\in\mathcal{U}$, there exists a unique positive function $\phi$ with $\phi-1\in C^{2,\alpha}_{\delta}(M)$ satisfying
\begin{align}\label{eq:Yamabe_eq}
-\frac{4n}{n-1}\,\Delta_g\phi\,+\,R_{g}\phi\,+\,n(n+1)\phi^{\frac{n+3}{n-1}}\,=\,0
\end{align}
and consequently, $\overline{g}=\phi^{\frac{4}{n-1}}g$ has constant scalar curvature $-n(n+1)$. Denoting this solution by $\phi_g$, we see that the pair 
\begin{align*}
(h,\psi)\,=\,\frac{d}{dt}(g_t,\phi_{g_t})|_{t=0}
\end{align*}
satisfies the linearized equation
\begin{equation}
\begin{split}
\label{eq:linear_yamabe}
P_g\psi&=\,\frac{4n}{n-1}\left(\frac{d}{dt}\Delta_{g+th}|_{t=0}\right)\phi_g-\left(\frac{d}{dt}R_{g_t}|_{t=0}\right)\phi_g\\
&=\,\frac{4n}{n-1}\bigg(
\frac{1}{2}\langle \nabla^g\mathrm{tr}_gh-2\mathrm{div}_gh,\nabla^g \phi_g\rangle_g-\langle h,\nabla^g d\phi_g\rangle_g\bigg)\\
&\,\,\,\quad+\phi_g\bigg(\Delta_g(\mathrm{tr}_gh)-\mathrm{div}_g(\mathrm{div}_gh)+\langle \mathrm{Ric}_g,h\rangle_g\bigg)\,,
\end{split}
\end{equation}
where
\begin{align*}
P_g\,=\,-\,\frac{4n}{n-1}\,\Delta_g\,+\,R_{g}\,+\,n(n+1)\,\frac{n+3}{n-1}\,\phi_g^{\frac{4}{n-1}}\,.
\end{align*}
Because the operator $-4\frac{n}{n-1}\Delta_g+R_g$ is known to be conformal and due to the obvious relation 
$(\phi_{e^{2u}g})^{\frac{4}{n-1}}=e^{-2u}(\phi_g)^{\frac{4}{n-1}}$, the operator $P_g$ is conformal as well in the sense that
\begin{align}\label{eq:conf_transformation}
P_{e^{2u}g}\big(\psi\big)\,=\,e^{-\frac{n+3}{2}u}\,P_g\big(e^{\frac{n-1}{2}u}\psi\big)\,.
\end{align}
For a metric $\overline{g}$ with $R_{\overline{g}}=-n(n+1)$, we have
\begin{align*}
P_{\overline{g}}\,=\,\frac{4n}{n-1}\Big(\!-\Delta_{\overline{g}}\,+\,(n+1)\Big)\,:\,C^{2,\alpha}_{\delta}(M)\to C^{0,\alpha}_{\delta}(M)\,.
\end{align*}
By \cite{Lee06}, this is an isomorphism for all $\delta\in (-1,n+1)$, in particular for $\delta$ as in the theorem. 
For a metric of the form $e^{2u}\overline{g}$ for some
 $u\in C^{2,\alpha}_{\delta}(M)$,  \eqref{eq:conf_transformation} implies that $P_{g}:C^{2,\alpha}_{\delta}(M)\to C^{0,\alpha}_{\delta}(M)$ is also an isomorphism.
Because all metrics  $g\in\mathcal{U}$ are of that form, \eqref{eq:linear_yamabe} implies the estimate
\begin{align*}
\left\|\frac{d}{dt}\phi_{g+th}|_{t=0}\right\|_{C^{2,\alpha}_{\delta}(M)}\!\leq\, C\left\|h\right\|_{C^{2,\alpha}_{\delta}(M; \,S^2T^*M)}
\end{align*}
with a constant $C>0$ depending locally uniformly continuously on $g\in\mathcal{U}$ with respect to the $C^{2,\alpha}_{\delta}$-topology.
Because $\mathcal{U}$ is a dense subset of $\mathcal{M}$,
the continuous map $\mathcal{U} \ni g\mapsto \phi_{g}-1\in C^{2,\alpha}_{\delta}(M)$, extends to a continuous map on $\mathcal{M}$ so that $\phi_g$ solves \eqref{eq:Yamabe_eq} and $\overline{g}=(\phi_g)^{\frac{4}{n-1}}g$ has constant scalar curvature $-n(n+1)$. Uniqueness of this solution is a straightforward application of the maximum principle applied to \eqref{eq:Yamabe_eq}.
\end{proof}

Now, we are able to present a proof of Theorem \ref{thm:PMT}

\begin{proof}[Proof of Theorem \ref{thm:PMT}]
Up to minor modifications, this approximation argument is very similar to the proof of \cite[Theorem 4.8]{DaKrMc}. Let $\varphi$ be an asymptotically hyperboloidal map of order $\delta>1$. 
We still denote by $b$ an arbitrary complete Riemannian metric that coincides with $\varphi^* g_{hyp}$ in a neighborhood of $\partial M$.
Being $g-b\in C^{2,\alpha}_{\delta}(M; S^2 T^*M)$, by Theorem \ref{Yamabe_lowreg} we find a unique asymptotically hyperboloidal metric $\ol{g}$ conformal to $g$ with scalar curvature $\Ro\equiv -6$. 
By \cite[Proposition 3.6 and Theorem E]{DaKrMc}, $m_{VR}(g)\geq m_{VR}(\ol{g})$. It therefore suffices to show that $m_{VR}(\ol{g})\geq 0$.

Let $g_i$ be a sequence of metrics on $M$ converging to $\ol{g}$ in $C^{2,\alpha}_{\delta}(M; S^2 T^*M)$ such that $g_i-b$ is supported in a $b$-geodesic ball of radius $r_i$ with $r_i\to+ \infty$. Then in particular, the metrics $g_i$ are polyhomogeneous. From \cite[Theorem 1.7]{AndChr} it follows the existence of a unique asymptotically hyperbolic metric $\ol{g}_i\in [g_i]$, of class $C^\infty$ and conformal to $g_i$, which satisfies $\Ro_{\ol{g}_i}\equiv -6$. Therefore, the metrics $\ol{g}_i$ are polyhomogeneous and by \cite[Proposition 4.3]{DaKrMc}, they are also asymptotically hyperboloidal.

By Section \ref{SectEst}, there exists the minimal positive Green function $\mathcal{G}_o$ (which depends on $\ol{g}_i$) that vanishes at infinity. To conclude that each $m_{VR}(\ol{g}_i)\geq 0$, we proceed as follows. Since each $\ol{g}_i$ is asymptotically Poincar\'e-Einstein according to the definition in \cite{DaKrMc}, by the proof of \cite[Proposition 2.6]{DaKrMc}, we can construct an asymptotically hyperboildal chart $\widetilde{\varphi}$ of order $\delta=2$ such that $\rho = e^{-r}$ defines a geodesic boundary defining function for $\widetilde{\varphi}_*\, \ol{g}_i$ associated with the representative $\frac{1}{4} g_{\SSS^n}$ in the conformal class at infinity. 
Then, expansion \eqref{estgreenfunct} follows from \eqref{eq:asymptotic_phg} and Theorem \ref{thm:PMT_0version} implies that $m_{VR}(\ol{g}_i)\geq 0$.

As $g_i\to \ol{g}$ in $C^{2,\alpha}_{\delta}(M; S^2 T^*M)$, \cite[Proposition 4.3]{DaKrMc} implies that $\ol{g}_i\to \ol{g}$ in $C^{2,\alpha}_{\delta}(M; S^2 T^*M)$. On metrics $g$ of constant scalar curvature $-6$, $-m_{VR}(g)$ coincides with the renormalized Einstein-Hilbert action $S(g)$ which was introduced in \cite[Section 3]{DaKrMc}. Since $g\mapsto S(g)$ is continuous with respect to the $C^{2,\alpha}_{\delta}$-topology, we therefore get $m_{VR}(\ol{g}_i)=-S(\ol{g}_i)\to -S(\ol{g})=m_{VR}(\ol{g})$. This implies 
$m_{VR}(\ol{g})\geq 0$, as desired.

Suppose that $m_{VR}(g) = 0$. Then, $g=\ol{g}$ by \cite[Proposition 3.6 and Theorem E]{DaKrMc}, and $\ol{g}$ is a critical point of the map $g\mapsto m_{VR}(g)$ on the manifold $\mathcal{C}$ of metrics of constant scalar curvature $-6$. By \cite[Corollary 4.4]{DaKrMc}, $\Ric_{\ol{g}}=-2\ol{g}$. Since we are in dimension three, $\ol{g}$ is of constant sectional curvature $-1$. The conclusion then follows from \cite[Theorem 6.9]{AndHow1998}.
\end{proof}

\bibliographystyle{amsplain}
\bibliography{biblio}

\end{document}